\documentclass{article}
\usepackage[a4paper,top=2.54cm,bottom=2.54cm,left=3.17cm,right=3.17cm]{geometry}
\usepackage[shortlabels]{enumitem}
\setenumerate[1]{itemsep=0pt,partopsep=0pt,parsep=\parskip,topsep=0pt}
\usepackage[tworuled,linesnumbered,noline,noend]{algorithm2e}
\usepackage{amsmath,amsthm,tikz,float,amssymb,tikz,mathtools,bbm,bm}
\usepackage[hidelinks]{hyperref}
\usepackage[numbers,sort&compress]{natbib}
\usepackage{algorithmic}

\allowdisplaybreaks

\newtheorem{theorem}{Theorem}[section]
\newtheorem{lemma}[theorem]{Lemma}
\newtheorem*{remark}{Remark}
\newtheorem{corollary}[theorem]{Corollary}

\newcommand{\opt}{\mathrm{OPT}}

\newcommand{\bp}{\mathbb{P}}

\newcommand{\bX}{\boldsymbol{X}}

\newcommand{\talg}{\mathrm{ALG}}

\newcommand{\ds}{\,\mathrm{d}s}
\newcommand{\dt}{\,\mathrm{d}t}
\newcommand{\dx}{\,\mathrm{d}x}
\usepackage{authblk}
\begin{document}
\title{\textbf{Online Trading as a Secretary Problem Variant}\thanks{This work was supported by the National Natural Science Foundation of China (Nos. 11971046, 12331014 and 71871009); Science and Technology Commission of Shanghai Municipality (No. 22DZ2229014).}}
\renewcommand\Authfont{\small} 
\renewcommand\Affilfont{\small\itshape}
\author[a]{Xujin Chen}
\author[a]{Xiaodong Hu}
\author[a]{Changjun Wang}
\author[b]{Yuchun Xiong}
\author[b]{Qingjie Ye}

\affil[a]{\footnotesize Academy of Mathematics and Systems Science, Chinese Academy of Sciences, Beijing 100190, China}
\affil[b]{\footnotesize School of Mathematical Sciences, Key Laboratory of MEA (Ministry of Education), Shanghai Key Laboratory of PMMP, East China Normal University, Shanghai 200241, China}
\affil[ ]{\textit {\{xchen, xdhu, wcj\}@amss.ac.cn, ycx@stu.ecnu.edu.cn, qjye@math.ecnu.edu.cn}}
\date{}
\maketitle

\begin{abstract}
    This paper studies an online trading variant of the classical secretary problem, called \emph{secretary problem variant trading} (SPVT), from the perspective of an intermediary who facilitates trade between a seller and $n$ buyers (collectively referred to as \emph{agents}). The seller has an item, and each buyer demands the item. These agents arrive sequentially in a uniformly random order to meet the intermediary, each revealing their valuation (price) of the item upon arrival. After each arrival, the intermediary must make an immediate and irrevocable decision (either buying from the seller or selling to a buyer) before the next agent appears. The intermediary's objective is to maximize the price of the agent who ultimately holds the item at the end of the process. As in the secretary problem, the intermediary (operating as an online algorithm) has no prior knowledge of agents' arrival order or prices. Note that the secretary problem can be viewed as a degenerate case of SPVT where the seller with zero price arrives first. We evaluate the performance of online algorithms for SPVT using two notions of competitive ratio: strong and weak.
    
    The \emph{strong} notion, directly inspired by the secretary problem, benchmarks the online algorithm against a powerful offline optimum: the highest price among the $n+1$ agents. We propose an online algorithm for SPVT achieving a strong competitive ratio of $\frac{4e^2}{e^2+1} \approx 3.523$, which is the best possible even when the seller's price may be zero. This tight ratio closes the gap between the previous best upper bound of $4.189$ and lower bound of $3.258$.
    
    In contrast, the \emph{weak} notion restricts the offline optimal algorithm to the given arrival order, although it still benefits from complete foreknowledge of the order and all agents' prices. The offline algorithm can no longer alter the predetermined arrival order to always place the item in the hands of the agent offering the highest price (as is permitted in the strong model). Against this weaker benchmark, we design a simple online algorithm for SPVT, achieving a weak competitive ratio of $2$. We further investigate the special case in which the seller's price is zero. For this special SPVT, we develop a double-threshold algorithm achieving a weak competitive ratio of at most $1.83683$ and establish a lower bound of $1.76239$. 

\end{abstract}

\section{Introduction}

The secretary problem \cite{ferguson1989who,gilbert1966recognizing,lindley1961dynamic,bruss1984unified} is a classic optimal stopping problem that has received considerable attention and research. The problem can be framed as scenarios where decision-makers try to sell a single good (item) for the highest possible price. Specifically, there are $n$ buyers who make offers to the decision-maker one after another, but the decision-maker has no prior knowledge of these offers and only learns the value of each one when it is made. Once a buyer makes an offer, the decision-maker must immediately decide whether to accept the offer, selling the item to that buyer. Once the decision to accept is made, it is final --- selling to a later buyer is no longer an option. Conversely, if the decision-maker declines the offer, they lose the opportunity to sell to that buyer permanently. 
Note that the decision-maker has no prior knowledge of the buyers' offers --- they only know that the offers come in a random order. At first glance, it might seem that as the number of buyers increases, identifying the ideal buyer would become more challenging, making it impossible to design an online algorithm with a constant competitive ratio. Surprisingly, however, the secretary problem has been shown to admit online algorithms that achieve constant competitive ratios.

\paragraph{Online trading model.} In this paper, we study the extensions of the secretary problem from a trading perspective, incorporating both buyers and a single \emph{seller}. In this setting, the decision-maker who initially holds no item acts as an intermediary between buyers and the seller. Specifically, $n$ buyers and a seller, collectively referred to as \emph{agents},  meet the intermediary (decision-maker) one by one in a completely uniform order. Each agent offers their valuation of the item to the intermediary as a price. The intermediary must decide immediately upon receiving the offer whether to trade with the current agent at the offered price --- either purchasing an item from the seller or selling an item to a buyer (provided the intermediary currently possesses one). Note that the trading model degenerates into the classical secretary problem when the seller arrives first and offers a price of zero. Thereby, we refer to this model as the \emph{secretary problem variant trading} (SPVT). 

\paragraph{Related work.} The most classic solution to the secretary problem employs a ``learn first, then decide'' strategy. Specifically, the decision-maker observes the first $\lfloor n/e \rfloor$ offers without accepting any of them. Subsequently, the decision-maker accepts the first offer that exceeds all preceding offers. As $n$ approaches infinity, this algorithm achieves a competitive ratio of $e$, meaning it secures at least $1/e$ of the expected value of the highest offer. This $e$ guarantee is asymptotically optimal for the secretary problem.
 
Koutsoupias and Lazos~\cite{koutsoupias2018online} introduced an online trading model that generalizes the SPVT by accommodating $m$ sellers while preserving the constraint that the intermediary's inventory is limited to a single item. They showed that no effective online algorithm exists for maximizing the intermediary’s gain from trade. For maximizing the social welfare -- the sum of prices of agents holding an item at the end of the trading process, the competitive ratio, studied in the literature, is \emph{strong} because it was measured against the \emph{strong} offline optimum --- the sum of $m$ highest prices among all $m+n$ agents --- which generally exceeds the achievable maximum social welfare. When $m=n$, Koutsoupias and Lazos~\cite{koutsoupias2018online} proposed an online algorithm with a strong competitive ratio of $1 + O(n^{-1/3} \log n)$.  Chen et~al.~\cite{chen2024algorithms}  generalized this balanced trading algorithm from \cite{koutsoupias2018online} to a general setting with large $m$ and $n$, achieving an asymptotic strong competitive ratio no greater than  $1 + O(m^{-1/3} \log m)$.  For single-buyer trading,  Chen et~al.~\cite{chen2024algorithms} developed an algorithm with the optimal strong competitive ratio of $1 + 1/m$. For single-seller trading, i.e., SPVT, they designed an algorithm with a strong competitive ratio of $4.189$, and proved that the best possible strong competitive ratio cannot be less than $3.258$. Due to the analytical complexity of the more realistic, agent-arrival-order-dependent offline optimum (referred to as the \emph{weak }offline optimum), the corresponding \emph{weak} competitive ratio is not yet well-studied even for SPVT, with few relevant results available. 

\paragraph{Other related work.} The secretary problem and their variants have remained popular research topics over time. When the 
task of selling one item is expanded to selling $k$ items with the goal of maximizing the total selling price, the problem transforms into the multiple-choice secretary problem \cite{kleinberg2005multiple}. This problem has been further generalized into the knapsack secretary problem \cite{babaioff2007knapsack} and the matroid secretary problem \cite{babaioff2007matroids,babaioff2018matroid}. 

The prophet inequality problem \citep{krengel1977semiamarts,krengel1978semiamarts,samuelcahn1984comparison} is another classic optimal stopping problem. The main difference between the secretary problem and the prophet inequality problem lies in the amount of information available to the decision-maker. In the secretary problem, the decision-maker has no prior knowledge of the buyers' offers --- they only know that the offers come in a random order. In contrast, in the prophet inequality problem, the order of buyers' offers is fixed, and the decision-maker knows the probability distribution of each buyer's offer before the process begins. The prophet secretary problem \cite{ehsani2018prophet,correa2021prophet} integrates characteristics from both the secretary problem and the prophet inequality problem. In this scenario, the decision-maker knows the price distribution of all buyers, who arrive randomly to make offers. When the prophet inequality problem is modified to allow the decision-maker to adjust the order of buyers' offers, it becomes the ordered selection prophet inequality problem \cite{peng2022order,bubna2023prophet}. In particular, when each buyer's offer follows the same known distribution, it is termed the independent and identically distributed prophet inequality problem \cite{hill1982comparisons,correa2017posted}, where the order of buyers' offers becomes irrelevant to the outcome. Additionally, multiple-choice prophet inequality problems \cite{hajiaghayi2007automated,alaei2014bayesian}, multiple-choice prophet secretary problems \cite{arnosti2023tight}, and matroid prophet inequality problems \cite{kleinberg2012matroid} have also been extensively investigated. Recent research has introduced sample-driven optimal stopping problems  \cite{correa2022prophet,correa2023sample}, where the decision-maker learns from historical data rather than knowing the exact price distribution. This problem establishes a link between the secretary problem and the independent and identically distributed prophet inequality problem.

Correa et al. \cite{correa2023trading} proposed a variant of online trading based on the prophet inequality problem. The objective is to maximize the decision-maker's profit by buying low and selling high in the market. The decision-maker's inventory is constrained to holding at most an item during the trading process. There are $n$ trading moments. At each trading moment, there is a market price at which the decision-maker can either buy or sell an item. Correa et al. \cite{correa2023trading}  proved that if the market prices are independent and identically distributed (i.i.d.)\ and the distribution is known to the decision-maker, then there exists an algorithm with a competitive ratio of 2, which is the best possible. Furthermore, when prices are i.i.d.\ but the distribution is unknown, they demonstrated the existence of an algorithm with a competitive ratio of $(2n-2)/(n-2)$. For non-i.i.d.\ prices, they proved that no effective algorithm exists.

The field has also seen attention on related models. For simultaneous arrival of buyers and sellers, Myerson and Satterthwaite~\cite{myerson1983efficient} considered the problem of designing trading mechanisms, proving the famous impossibility result: it is impossible to design a Bayesian incentive-compatible and individually rational mechanism that always allocates items to the party with higher price. Recently, research on approximation mechanisms has made significant progress, including cases for maximizing social welfare \cite{blumrosen2014reallocation} and profit \cite{blumrosen2016approximating}. The online portfolio problem \cite{online2014li}, which focuses on allocating funds among assets to maximize investment returns, has also been widely studied. 

\paragraph{Our contribution.} 
In Section~\ref{sec:pre}, we present the mathematical model for the SPVT and formally define the strong and weak competitive ratios, which compare the solution output by the online algorithm with the strong and weak offline optimum, respectively. In Section \ref{sec:strong}, we study online algorithm design for SPVT to minimize the strong competitive ratio. We propose a novel algorithm (Algorithm~\ref{alg:our}) that achieves a strong competitive ratio of $\frac{4e^2}{e^2+1} \approx 3.523$ (Theorem~\ref{thm:tight}), which is proved to be tight (Theorem~\ref{thm:lower}). 
This closes the gap between the prior upper bound 4.189 and lower bound 3.258 \cite{chen2024algorithms}.  
In Section \ref{sec:weak}, we design online algorithms with small weak competitive ratios.
We first present an algorithm (Algorithm~\ref{alg:our2}) that achieves a weak competitive ratio of 2 (Theorem~\ref{thm:algweak}). For the special case where the seller's price is zero, we further develop a double-threshold algorithm (Algorithm~\ref{alg:our3}), which achieves a weak competitive ratio of at most $1.83683$ (Theorem~\ref{thm:zeroupper}). Moreover, we establish a lower bound of $1.76239$ for the weak competitive ratio of this special SPVT (Theorem~\ref{thm:weaklower}), which automatically applies to the general SPVT. See Table~\ref{table:compare} for a summary of our results and comparison with previous work \cite{chen2024algorithms} on SPVT.

\begin{table}[H]
\caption{Competitive ratios on the SPVT}\label{table:compare}
\centering
\begin{tabular}{cccc}
\hline
Seller's price \,\,& Competitive ratio type \,\, & Upper bound \,\, & Lower bound                                    \\\hline

Any or 0   & Strong             &    3.523,\;
 4.189 \cite{chen2024algorithms}    & \;\;3.523,\;
 3.258 \cite{chen2024algorithms}                                  \\

Any & Weak  &   \hspace{0mm}2,\hspace{6mm}---\,\;   & $1.76239$,\hspace{6mm}---\;\;\;\;\;\;\;\;\;     \\

0 & Weak  & $1.83683$,\hspace{6mm}---\;\;\;\;\;\;\;\;\;\;\; & $1.76239$,\hspace{6mm}---\;\;\;\;\;\;\;\;\; \\ 
\hline
\end{tabular}
\end{table}

To highlight the technical novelty and challenges in our algorithm design and analysis, we briefly mention some key connections and distinctions between our work and previous related results and methods. 
Our SPVT resembles the classic secretary problem, but the online trading problem must balance between two possibilities: whether the highest price is offered by the seller or by a buyer. The core task is to decide whether to buy the item from the (unique) seller based on the timing of their offer and the number of prices observed. Once the algorithm decides not to buy from the seller, the process terminates immediately. Conversely, if the item is bought from the seller, the remaining task reduces to selling it to the highest buyer, effectively transforming the problem into the classical secretary problem. We follow the framework of Chen et~al.~\cite{chen2024algorithms} and adopt a similar approach, but with a more refined consideration of details. 

The key difference between our approach and that of \cite{chen2024algorithms} lies in how we treat the seller. Unlike the previous approach, in this paper, we do not consider the seller as a fictitious buyer. To be more specific, from the perspective of the offline optimal solution, the seller's arrival is equivalent to the intermediary obtaining an item from an imaginary charity, and ``intermediary not buying from the seller'' in the original trading setting is equivalent to ``selling the item to the seller'' (who is thus a fictitious buyer) in the imaginary setting. However, from the perspective of designing an online algorithm, these two situations are not equivalent: the algorithm has full control over keeping the item with the seller but does not have full control over selling it to the highest-price buyer. We exploit this subtle yet important distinction to design online algorithms for SPVT that yield a tight strong competitive ratio and small weak competitive ratios. Furthermore, we successfully establish a tight bound for the strong competitive ratio by utilizing a linear programming framework and Ramsey's theorem on infinite hypergraphs.

\section{Preliminaries}\label{sec:pre}
We formalize the model of SPVT (Secretary Problem Variant Trading). Buyers'  non-negative prices (valuations) for the single item are denoted as $X_1, \dots, X_n$, while the seller's price is denoted as $X_{n+1}$. The buyers and seller, collectively referred to as \emph{agents}, arrive at the intermediary according to a uniformly random permutation. We order agents using their arrival indices. The \emph{$i$-th agent} (resp.\ \emph{$i$-th buyer}) is the one whose arrival is preceded by the arrivals of exactly $i-1$ agents  (resp.\ buyers). The $i$-th agent is \emph{earlier} or \emph{precedes} the $j$-th agent if and only if $i<j$.

When the $i$-th agent arrives, they submit the value $X_{\sigma(i)}$ as a price to the intermediary, where $\sigma \in S_{n+1}$ is a permutation on the set $[n+1] = \{1, 2, \dots, n+1\}$, mapping the arrival index $i$ to the agents' identity $\sigma(i)$. 

\begin{itemize}
    \item The seller starts with one item. Neither the buyers nor the intermediary initially hold any items. Each buyer demands one item. 
    \item If the $i$-th agent is a buyer, i.e., $\sigma(i) \leq n$, they offer to purchase one item from the intermediary at the price $X_{\sigma(i)}$. Let the random variable $Y_{\sigma(i)}$ indicate the transaction outcome: $Y_{\sigma(i)} = 1$ if the purchase is successful, and $Y_{\sigma(i)} = 0$ otherwise.
    \item If the $i$-th agent is the seller, i.e., $\sigma(i) = n+1$, they offer to sell the item to the intermediary at the price $X_{\sigma(i)}$. The random variable $Y_{\sigma(i)}$ indicates the outcome: $Y_{\sigma(i)} = 1$ if the sale is successful, and $Y_{\sigma(i)} = 0$ otherwise.
\end{itemize}

Whether the agent is a buyer or the seller, the intermediary must immediately make an irrevocable decision about whether to trade with the agent. The intermediary can only sell an item to a buyer if it possesses one, and can only purchase an item from a seller if it doesn't currently hold one. 

\paragraph{Optimization objectives.} The SPVT problem requires the intermediary to make real-time feasible decisions  $Y_{\sigma(i)}\in\{0,1\}$ for each arriving agent $\sigma(i)$, $i = 1, \dots, n+1$, where $Y_{\sigma(i)}=1$ indicates making a deal and $Y_{\sigma(i)}=0$ means passing, 
to maximize the social welfare:
\[
    \max \,\mathbb{SW}=\sum_{i=1}^n X_i Y_i +  X_{n+1} (1 - Y_{n+1}).
\]
In other words, the intermediary wants the agent who ends up holding the item to have the highest valuation, where $X_i$ represents the item's utility to the $\sigma^{-1}(i)$-th agent.

\paragraph{Offline optimum.} The competitive ratio of an online algorithm, which makes the decision on behalf of the intermediary, is the maximum (worst-case) ratio between the offline optimal value and the objective value of the algorithm’s output. To clearly define the competitive ratio, it is necessary to specify how the corresponding offline optimization problem is defined. 
Typically, there are two choices:
\begin{itemize}
    \item \emph{Strong offline optimum}: The offline algorithm assumes that agents arrive in the order that maximizes the objective value. This maximum value is $\max_{i=1}^n X_{i}$, called the strong offline optimum.
    \item \emph{Weak offline optimum}: Under the assumption that agents arrive randomly, the arrival order is given at the start to the offline algorithm, which makes the best decision in response to this given order.  The optimal objective value is an expectation that is taken over all $(n+1)!$ random orders.
\end{itemize}

The strong offline optimum is generally larger than the weak one in SPVT (although they are the same for the classical secretary problem). The competitive ratio based on the strong offline optimum is called the \emph{strong competitive ratio}, and the one based on the weak offline optimum is termed the \emph{weak competitive ratio}.

\section{Strong competitive ratios}\label{sec:strong}
Throughout this section, we focus exclusively on online algorithm design for minimizing the strong competitive ratio. For brevity, all references to offline optimal solutions and competitive ratios pertain to their strong versions, and we will often omit explicit mention of ``strong'' henceforth within this section.

Given the price vector $\bX = (X_1, X_2, \dots, X_{n+1})$, we assume all values are distinct, employing a universal tie-breaking rule if necessary. The buyers are labeled such that their values are sorted in descending order: $X_1 > X_2 > \dots > X_n$. Clearly, the offline optimal value is 
\[\opt(\bX) = \max\{X_1, X_2, \dots, X_{n+1}\} = \max\{X_1, X_{n+1}\}.\]
For an algorithm $\talg$ applied to the SPVT, let $p_i^{\talg}(\bX)$ denote the probability that the item remains with agent $i$ in the end. 
The expected social welfare achieved by the algorithm is  \[\talg(\bX) = \sum_{i=1}^{n+1} X_i p_i^{\talg}(\bX).\]
Let $\rho_n = \min_{\talg} \max_{\bX} \frac{\opt(\bX)}{\talg(\bX)}$ denote the best strong competitive ratio of SPVT with $n$ buyers. Chen et~al.~\cite{chen2024algorithms} established that
$3.258 \le \lim_{n \to \infty} \rho_n \le 4.189$.
In this section, we will prove that
\[
\rho_n \le \lim_{\eta \to \infty} \rho_\eta = \frac{4e^2}{e^2+1} \approx 3.523.
\]

We break the proof into two parts: Section~\ref{sec:upper} presents a random online algorithm (Algorithm~\ref{alg:our}) with strong competitive ratio at most $\frac{4e^2}{e^2+1}$, providing an upper bound for $\rho_n$ (Corollary~\ref{cor:upper}), while Theorem~\ref{thm:lower} in Section~\ref{sec:lower} establishes the matching lower bound $\frac{4e^2}{e^2+1}-o(1)$. The lower bound holds even when the seller's price is 0.

\subsection{Random transaction: upper bounding the strong ratio} 
\label{sec:upper}
To solve the SPVT, we employ a randomization technique that was first introduced by Bruss~\cite{bruss1984unified} to handle the classical secretary problem. The most challenging part of our algorithm design (Algorithm~\ref{alg:our}) is how the intermediary decides to buy the item from the seller, which has not been addressed in previous studies of classical counterparts.

\begin{algorithm}[ht]
    \caption{Random Transaction for SPVT}
    \label{alg:our}
    Let $T_1, T_2, \dots, T_{n+1}$ be $n+1$ independent random numbers uniformly distributed on $[0,1]$, where $T_{1} < T_{2} < \dots < T_{{n+1}}$ is assumed (via employing number rearrangement and a universal tie-breaking rule). Consider $T_i$ as the specific \emph{arrival time} of the $i$-th agent for all $i\in[n+1]$.

   \textbf{If} the seller arrives at a time later than $(e-1)/e$ and the seller's price is higher than all agents preceding it, \textbf{then} the intermediary skips over all agents (the algorithm stops), \textbf{else} the intermediary buys the item from the seller.

After the time the intermediary buys from the seller,    let $b$ be the earliest buyer satisfying the following conditions: $b$ arrives at a time after $1/e$, and $b$'s price is higher than all agents preceding $b$. 
\textbf{If} such a buyer $b$ exists, \textbf{then} the intermediary sells the item to $b$.
\end{algorithm}

\begin{theorem}\label{thm:tight}
    The strong competitive ratio for Algorithm \ref{alg:our} is at most $(4e^2)/(e^2+1)\approx 3.523$.
\end{theorem}
\begin{proof} Recalling $\opt(\bX) = \max\{X_1, X_{n+1}\}$, if $X_{n+1}> X_1$, then the item is ultimately held by the seller if and only if the seller's arrival time is later than $(e-1)/e$. Since the arrival time of the seller is uniform on $[0,1]$, we have
   $p_{n+1}^{\talg}(\bX)=1-\frac{e-1}{e}=\frac{1}{e}$, which gives
    \[\frac{\opt(\bX)}{\talg(\bX)}\le \frac{X_{n+1}}{X_{n+1}\cdot p_{n+1}^{\talg}(\bX)}=e<\frac{4e^2}{e^2+1}.\]

    It remains to consider the case of  $X_{n+1}< X_1$. Suppose that there are exactly $\mu\,(\ge1)$ buyers whose prices are higher than the seller's. Recall that $X_1 > X_2 > \dots > X_n$. We have  $X_\mu>X_{n+1}$ and either $\mu=n$ or $X_{n+1}> X_{\mu+1}$. Let $t_{n+1}$ denote the arrival time of the seller, and $t_1$ denote the arrival time of buyer 1 (the buyer with the highest price). For any time point $t\in[0,1]$, let $B_t=\{i:2\le i\le \mu$,  buyer $i$  arrives before $t\}$ consists of buyers who offer a higher price than the seller and arrive earlier than time $t$. When $B_t\neq \emptyset$, for any $s\in[0,t]$, let $A_{s,t}$ denote the event that the \emph{best} buyer (the one with the highest price) in $B_t$ arrives before time $s$. 
    
    If $t_{n+1}\le (e-1)/e$, which implies that the intermediary buys from the seller, then SW achieved by the algorithm is $X_1$ if and only if $t_1> \max\{t_{n+1},1/e\}$ and either $B_{t_1}=\emptyset$ or the best buyer in $B_{t_1}$ arrives before time $\max\{t_{n+1},1/e\}$. If $t_{n+1}>(e-1)/e$, then the algorithm attains  $\text{SW}=X_1$ if and only if $t_1> t_{n+1}$ (buyer 1 arrives after the seller), $B_{t_{n+1}}\neq \emptyset$ (some preceding buyer offers a higher price than the seller) and $A_{t_{n+1},t_1}$ happens (the best buyer in $B_{t_1}$ arrives before $t_{n+1}$). Therefore, using variables $s$ and $t$ to represent $t_{n+1}$ and $t_1$, respectively, we obtain
    \begin{align*}
        p_1^{\talg}(\bX)=&\int_{0}^{1/e} \ds \int_{1/e}^{1} (\mathbb{P}(B_t= \emptyset)+\mathbb{P}(B_t\neq \emptyset) \cdot \mathbb{P}(A_{1/e,t}|B_t\neq \emptyset))\,\dt\\
        &+\int_{1/e}^{(e-1)/e} \ds \int_{s}^{1} (\mathbb{P}(B_t= \emptyset)+\mathbb{P}(B_t\neq \emptyset) \cdot \mathbb{P}(A_{s,t}|B_t\neq \emptyset))\,\dt\\
        &+\int_{(e-1)/e}^{1} \mathbb{P}(B_s\neq \emptyset)\,\ds \int_{s}^{1} \mathbb{P}(A_{s,t}|B_s\neq \emptyset)\,\dt.
    \end{align*}

    By definition, we have $\mathbb{P}(B_t= \emptyset)=\prod_{i=2}^{\mu} \mathbb{P}(\text{buyer $i$ arrives after time $t$})=(1-t)^{\mu-1}$. Since the arrival time of the best buyer in $B_t$ is uniform on $[0,t]$, we have $\mathbb{P}(A_{s,t}|B_t\neq \emptyset)=s/t$. Note that the event $(B_s\neq \emptyset)\cap A_{s,t}$ happens if and only if so does $(B_t\neq \emptyset)\cap A_{s,t}$, Thus we have $\mathbb{P}(B_s\neq \emptyset) \cdot \mathbb{P}(A_{s,t}|B_s\neq \emptyset)=\mathbb{P}(B_t\neq \emptyset) \cdot \mathbb{P}(A_{s,t}|B_t\neq \emptyset)$, giving
    \begin{align*}
        p_1^{\talg}(\bX)=&\int_{0}^{1/e} \ds \int_{1/e}^{1} \left((1-t)^{\mu-1}+\frac{1-(1-t)^{\mu-1}}{et}\right)\,\dt\\
        &+\int_{1/e}^{(e-1)/e} \ds \int_{s}^{1} (1-t)^{\mu-1}\,\dt\\
        &+\int_{1/e}^{1} s\,\ds \int_{s}^{1} \frac{1-(1-t)^{\mu-1}}{t}\,\dt.
    \end{align*}

For convenience, we write the above equation as $p_1^{\talg}(\bX)=\delta_\mu=\alpha_\mu+\beta_\mu+\gamma_\mu$, where
    \begin{align*}
        \alpha_\mu&=\int_{0}^{1/e} \ds \int_{1/e}^{1} \left((1-t)^{\mu-1}+\frac{1-(1-t)^{\mu-1}}{et}\right)\,\dt\\
        \beta_\mu&=\int_{1/e}^{(e-1)/e} \ds \int_{s}^{1} (1-t)^{\mu-1}\,\dt\\
        \gamma_\mu&=\int_{1/e}^{1} s\,\ds \int_{s}^{1} \frac{1-(1-t)^{\mu-1}}{t}\,\dt.
    \end{align*}

Notice that
    \begin{align*}
        \alpha_\mu-\alpha_{\mu+1}&=\int_{0}^{1/e} \ds \int_{1/e}^{1} \left(t(1-t)^{\mu-1}-\frac{(1-t)^{\mu-1}}{e}\right)\,\dt\\
        &=\frac{1}{e}\cdot \int_{0}^{(e-1)/e} \left(\frac{e-1}{e}\cdot t^{\mu-1}-t^\mu\right)\,\dt\\
        &= \frac{1}{e\mu(\mu+1)}\cdot \left(\frac{e-1}{e}\right)^{\mu+1},\\
        \beta_\mu-\beta_{\mu+1}&=\int_{1/e}^{(e-1)/e} \ds \int_{s}^{1} t(1-t)^{\mu-1}\,\dt\\
        &=\int_{1/e}^{(e-1)/e} \ds \int_{0}^{s} (t^{\mu-1}-t^\mu)\,\dt\\
        &=\int_{1/e}^{(e-1)/e} \left(\frac{s^\mu}{\mu}-\frac{s^{\mu+1}}{\mu+1}\right)\,\ds\\
        &=\frac{1}{\mu(\mu+1)(\mu+2)}\left(\left(\frac{e-1}{e}\right)^{\mu+1}\cdot \frac{2e+\mu}{e}-\frac{1}{e^{\mu+1}}\cdot \frac{2e+\mu e-\mu}{e}\right),\\
        \gamma_\mu-\gamma_{\mu+1}&=-\int_{1/e}^{1} s\,\ds \int_{s}^{1} (1-t)^{\mu-1}\,\dt\\
        &=-\int_{0}^{(e-1)/e} (1-s)\,\ds \int_{0}^{s} t^{\mu-1}\,\dt\\
        &=-\int_{0}^{(e-1)/e} \frac{s^\mu-s^{\mu+1}}{\mu}\,\ds\\
        &=-\frac{\mu+1+e}{e\mu(\mu+1)(\mu+2)}\cdot \left(\frac{e-1}{e}\right)^{\mu+1},
    \end{align*}
   and
    \begin{align*}
        \delta_\mu-\delta_{\mu+1}&=\frac{1}{e^{\mu+2}\mu(\mu+1)(\mu+2)}\cdot \left((\mu+1+e)\cdot (e-1)^{\mu+1}-(2e+\mu e-\mu)\right)\\
        &\ge \frac{2e+\mu e-\mu}{e^{\mu+2}\mu(\mu+1)(\mu+2)}\cdot ((e-1)^{\mu}-1)\\
        &>0.
    \end{align*}
    where the first inequality follows because $(e-1)(e+1)\approx 6.389>5.437\approx 2e$. It implies that $\{\delta_\mu\}$ is a decreasing sequence, and thus
    \[p_1^{\talg}(\bX)=\delta_\mu\ge \lim_{\nu\rightarrow +\infty} \delta_\nu=\int_{0}^{1/e} \ds \int_{1/e}^{1} \frac{1}{et}\,\dt+\int_{1/e}^{1} s\,\ds \int_{s}^{1} \frac{1}{t}\,\dt=\frac{e^2+1}{4e^2}.\]
    It follows that the competitive ratio of the algorithm is 
    \[\frac{\opt(\bX)}{\talg(\bX)}\le \frac{X_1}{X_1\cdot p_1^{\talg}(\bX)}\le \frac{4e^2}{e^2+1}.\qedhere\]
\end{proof}

Instantly, we have $\frac{4e^2}{e^2+1}$ as an upper bound on the best strong competitive ratio $\rho_n$ for the SPVT.

\begin{corollary}\label{cor:upper}
    $\rho_n\le \frac{4e^2}{e^2+1}\approx 3.523$ for all $n\in\mathbb{N}$.
\end{corollary}
\subsection{LP formulation: lower bounding the strong ratio}\label{sec:lower}

To establish a lower bound on the best strong competitive ratio, we may restrict our attention to a subset of instances. Any lower bound established for the subset is certainly a lower bound for the general problem.
In this subsection, we only consider the instances where \emph{the seller consistently offers a price of $0$}, for which w.l.o.g.\ the intermediary always buys from the seller. 

Consequently (with a slight abuse of notation), the agent price vector is simplified to the buyer price vector $\bX = (X_1, \dots, X_n)$, where $X_1 > X_2 > \dots > X_n$. The offline optimal value is then $\opt(\bX) = X_1$. Each algorithm $\talg$, for the SPVT, is associated with a family $\mathcal F_{\talg}$ of functions $r_{i,j}:\mathbb{R}_+^j\rightarrow [0,1]$, $1\le i\le j\le n$, where  $r_{i,j}(x_1,\dots,x_j)$ is the probability that the algorithm \emph{stops} at the $j$-th buyer, i.e., the intermediary sells the item to the $j$-th  buyer, given that the seller is the $i$-th agent, the intermediary holds an item when the $j$-th buyer arrives, and  $x_k$ is the price offered by the $k$-th buyer for all $1\le k\le j$. 

While we have simplified the problem instances, the range of feasible algorithmic operations remains vast and complex to analyze. Therefore, we introduce restrictions on the algorithms themselves to facilitate analysis, while demonstrating that these restrictions do not affect the best achievable competitive ratio. Our objective is to prove that restricting algorithms to decisions based solely on whether the current agent's price is the highest discovered so far does not change the best competitive ratio. 

To formalize this, we introduce the notion of value-independence. 
Given $\epsilon > 0$, $V \subseteq \mathbb{N}$, and $1 \le i \le j \le n$, an algorithm $\talg $ with $\mathcal F_{\talg}= \{r_{i',j'}: 1 \le i' \le j' \le n\}$ is $(\epsilon, i, j)$-\emph{value-independent} on $V$ if for any $k\in [j]$, there exists $q_{i,j,k} \in [0,1]$ such that for any distinct $x_1, \dots, x_j \in V$ with $k=|\{x_l:x_l\ge x_j\}|$, we have $r_{i,j}(x_1, \dots, x_j) \in [q_{i,j,k} - \epsilon, q_{i,j,k} + \epsilon]$.  We say that the algorithm $\talg$ is $\epsilon$-\emph{value-independent} on $V$ if it is $(\epsilon, i, j)$-value-independent on $V$ for all $1 \le i \le j \le n$. The following lemma shows that we only need to consider value-independent algorithms.

\begin{lemma}\label{lma:evo}
   For any $\epsilon>0$ and any algorithm $\talg$, there exists an  algorithm $\talg'$ and infinite set $V\subseteq \mathbb{N}$ such that $\talg'$ is $\epsilon$-value-independent on $V$, and $\talg(\bX)=\talg'(\bX)$ for all $\bX\subseteq V^n$.
\end{lemma}
While the correctness of Lemma~\ref{lma:evo} is quite intuitive, its formal proof is quite involved. The high-level proof idea draws inspiration from \cite{correa2022prophet}, utilizing Ramsey's theorem on infinite hypergraphs as a tool to narrow down the price vector space. However, adapting this technique to our setting requires careful handling of additional structural details. Consequently, we defer the proof of Lemma~\ref{lma:evo} to Appendix~\ref{apx:lma:evo}.

Inspired by \cite{buchbinder2014secretary,correa2023sample}, we formulate the search for an optimal value-independent strategy as the following linear program.
\begin{alignat}{4}
    & \max         & \quad & \sum_{i=1}^{n} \sum_{j=i}^{n} \frac{1}{n+1}\cdot \frac{j}{n}\cdot x_{i,j}  & \quad & \label{obj:discrete}\\
    & \text{ s.t.} &       & j\cdot x_{i,j}\le 1-\sum_{k=i}^{j-1} x_{i,k}  &       &\forall\, 1 \le i \le j \le n \label{upperq:discrete}\\
    &              &       & x_{i,j}\ge 0   &       &\forall\, 1 \le i \le j \le n    \label{prob:discrete}
\end{alignat}

\begin{lemma}\label{lma:pb1}
   Given any $0$-value-independent algorithm $\talg$ on $V$ and price vector $\bX\in V^n$, for all $1\le i\le j\le n$, let $C_i$ be the event that the seller is the $i$-th agent, and let $D_j$ be the event that under $\talg$ and $\bX$ the intermediary sells the item to the $j$-th buyer whose price is higher than all preceding buyers. Let $x_{i,j}=\mathbb{P}(D_j|C_i)$, $1\le i\le j\le n$. Then $x$ is a feasible solution to the linear program \eqref{obj:discrete}--\eqref{prob:discrete}, and its corresponding objective value, defined in  \eqref{obj:discrete}, equals $p_1^{\talg}(\bX)$.
\end{lemma}
\begin{proof}
    Let $D'_j$ be the event that the $j$-th  buyer offers a higher price than all the preceding buyers. Let $D''_j$ be the event that the intermediary has an item when the $j$-th buyer arrives. Then for all $i\le j$,
    \[x_{i,j}=\mathbb{P}(D_j|C_i)\le \mathbb{P}(D'_j\cap D''_j|C_i)=\mathbb{P}(D'_j|C_i)\cdot \mathbb{P}(D''_j|C_i)\le \frac{1}{j}\cdot \left(1-\sum_{k=i}^{j-1} x_{i,k}\right),\]
    which implies that $x$ satisfies \eqref{upperq:discrete}.

    Let $E_j$ be the event that the $j$-th  buyer offers the highest price among all $n$ buyers. Then for all $i\le j$, by the 0-value independence, we have
    \[\mathbb{P}(E_j|C_i\cap D_j)=\mathbb{P}(E_j|C_i\cap D'_j)=\frac{\mathbb{P}(E_j|C_i)}{\mathbb{P}(D'_j|C_i)}=\frac{1/n}{1/j}=\frac{j}{n}.\]
    Therefore,
    \[p_1^{\talg}(\bX)=\sum_{i=1}^{n} \mathbb{P}(C_i) \sum_{j=i}^{n} \mathbb{P}(E_j|C_i\cap D_j)\cdot \mathbb{P}(D_j|C_i)=\sum_{i=1}^{n} \frac{1}{n+1} \sum_{j=i}^{n} \frac{j}{n}\cdot x_{i,j}.\qedhere\]
\end{proof}

To obtain the final bound, we evaluate the optimal value of this program through its dual formulation.

\begin{lemma}\label{lma:linear}
    The optimal value of the linear program \eqref{obj:discrete}--\eqref{prob:discrete} does not exceed $(e^2+1)/(4e^2) + o(1)$.
\end{lemma}
\begin{proof}
To provide an upper bound for the optimal value of the linear program \eqref{obj:discrete}--\eqref{prob:discrete}, we consider its dual linear program:  
\begin{alignat}{4}
& \min         & \quad & \sum_{i=1}^{n} \sum_{j=i}^{n} y_{i,j}  & \quad & \label{obj:dual}\\
& \text{ s.t.} &       & j\cdot y_{i,j}+\sum_{k=j+1}^{n} y_{i,k}\ge \frac{j}{(n+1)\cdot n}  &       &\forall\, 1 \le i \le j \le n \label{upperq:dual}\\
&              &       & y_{i,j}\ge 0   &       &\forall\, 1 \le i \le j \le n    \label{prob:dual}
\end{alignat}
    
By the duality theory of linear programming, it suffices to prove that the dual program \eqref{obj:dual}--\eqref{prob:dual} has a feasible solution with an objective value of $(e^2+1)/(4e^2)+o(1)$.  
We construct a sequence $a_j$, $j=1,2,\ldots,n$ as follows:  
\[
a_j = \frac{1}{(n+1)\cdot n}\left(1 - \sum_{k=j}^{n-1} \frac{1}{k}\right).
\]  

We claim that $y_{i,j} = \max\{a_j, 0\}$ is a feasible solution to the dual program \eqref{obj:dual}--\eqref{prob:dual}. Clearly, $y_{i,j}$ satisfies \eqref{prob:dual}. To verify the validity of \eqref{upperq:dual}, we note that  
\[
\begin{aligned}
    (n+1)\cdot n \cdot \left(j \cdot y_{i,j} + \sum_{k=j+1}^{n} y_{i,k}\right) 
    &\geq j - \sum_{l=j}^{n-1} \frac{j}{l} + \sum_{k=j+1}^{n} \left(1 - \sum_{l=k}^{n-1} \frac{1}{l}\right) \\
    &= j + (n-j) - \sum_{l=j}^{n-1} \frac{j + (l-j)}{l} \\
    &= j.
\end{aligned}
\]
So $y_{i,j}$ satisfies \eqref{upperq:dual}, establishing the feasibility.  

Let $j^* = \lceil 1 + (n-1)/e \rceil$. For any $j \geq j^*$,  by the definition of the definite integral,  
$ 1 - \sum_{k=j}^{n-1} \frac{1}{k} \geq 1 - \int_{j-1}^{n-1} \frac{1}{x} \, \mathrm{d}x = 1 + \ln \frac{j-1}{n-1} \geq 0,
$
which implies $a_j \geq 0$ and $y_{i,j}=a_j$ for all $j\ge j^*$.  Hence, for the objective value in \eqref{obj:dual}, we have  
\begin{align*}
    \sum_{i=1}^{n} \sum_{j=i}^{n} y_{i,j}&=\sum_{j=1}^{n} \sum_{i=1}^{j} y_{i,j}\\
    &\ge \sum_{j=j^*}^{n} j\cdot a_{j}\\
    &=\frac{1}{(n+1)\cdot n}\cdot \sum_{j=j^*}^{n} j\cdot \left(1-\sum_{k=j}^{n-1}\frac{1}{k}\right)\\
    &\ge \frac{1}{n+1} \sum_{j=j^*}^{n} \frac{j}{n}\cdot \left(1+\ln \frac{j-1}{n-1}\right)\\
    &=\int_{1/e}^{1} x(1+\ln x)\dx +o(1)\\
    &=\frac{e^2+1}{4e^2}+o(1).\qedhere
\end{align*}
\end{proof}

Lemmas \ref{lma:evo} -- \ref{lma:linear} establish a lower bound of $\frac{4e^2}{e^2+1}-o(1)$ on the strong competitive ratio. Notably, this bound holds even when the seller's price is fixed at $0$.
\begin{theorem} \label{thm:lower}
   For any algorithm $\talg$ of the SPVT and any $\delta>0$, there exists $n_0$ such that, for all $n\ge n_0$, there exists a price vector ${\bX}'\in\mathbb{N}^n$ satisfying 
    \[\frac{\opt(\bX')}{\talg({\bX}')}\ge \frac{4e^2}{e^2+1}-\delta.\]
\end{theorem}
\begin{proof}
    By Lemma \ref{lma:evo}, there exists an algorithm $\talg'$ with $\mathcal F_{\talg'}=\{r_{i,j}:1\le i\le j\le n\}$ and infinite set $V\subseteq \mathbb{N}$ such that $\talg'$ is $\epsilon$-value-oblivious on $V$, and $\talg(\bX)=\talg'(\bX)$ for all $\bX\subseteq V^n$. By definition, for all $1\le i\le j\le n$, there exists $q_{i,j,1}\in [0,1]$ such that for all pairwise distinct $x_1,\dots,x_j\in V$ with $x_j> \max\{x_1,\dots,x_{j-1}\}$, it holds that $r_{i,j}(x_1,\dots,x_j)\in [q_{i,j,1}-\epsilon,q_{i,j,1}+\epsilon]$. We can define a 0-value-oblivious algorithm $\talg''$ on $V$ by $\mathcal F_{\talg''}=\{r'_{i,j}:1\le i\le j\le n\}$ such that $r'_{i,j}(x_1,\dots,x_j)=q_{i,j,1}$ if $x_j> \max\{x_1,\dots,x_{j-1}\}$ and $r'_{i,j}(x_1,\dots,x_j)=0$ otherwise.
    
    Since $V$ is an infinite subset of $\mathbb{N}$, we can select $n$ distinct numbers $X'_1, X'_2, \dots, X'_n \in V$ such that $X'_1 > nX'_2$ and $X'_2 > X'_3 > \dots > X'_n$. Taking $\bX'=(X'_1,X'_2,\ldots,X'_n)$, by Lemmas \ref{lma:pb1} and \ref{lma:linear}, we have 
    \[p_1^{\talg''}(\bX') \le (e^2+1)/(4e^2) + o(1).\] 
    Let $\epsilon = 1/n^2$. Since $|p_1^{\talg'}(\bX') - p_1^{\talg''}(\bX')| \le n\epsilon$, it follows that 
    \[p_1^{\talg'}(\bX') \le (e^2+1)/(4e^2) + o(1).\] 
    Therefore,

    \[\frac{\opt(\bX')}{\talg(\bX')}\ge \frac{X'_1}{X'_1\cdot p_1^{\talg'}(\bX')+X'_2}=\frac{1}{p_1^{\talg'}(\bX')}-o(1)=\frac{4e^2}{e^2+1}-o(1).\qedhere\]
\end{proof}

It follows from Theorem~\ref{thm:lower} that $ \lim_{n \to \infty} \rho_n \ge \frac{4e^2}{e^2+1}$, which in combination with Corollary~\ref{cor:upper} yields 
\[\rho_n \le \lim_{\eta \to \infty} \rho_\eta = \frac{4e^2}{e^2+1} \approx 3.523,\] 
and shows that Algorithm~\ref{alg:our} achieves a strong competitive ratio that is the best possible.

\section{Weak competitive ratios}\label{sec:weak}
Having established the tight strong competitive ratio in the previous section, we now shift our focus to algorithm design for minimizing the weak competitive ratio. Consistent with the convention used earlier, throughout this section, unless explicitly stated otherwise, all references to offline optima and competitive ratios will pertain to their weak versions. We will often omit the modifier ``weak'' for brevity within this section.

First, let us first provide the formula for the weak offline optimal value $\opt(\bX)$. Let $\bX = (X_1, X_2, \dots, X_{n+1})$ be the price vector with $X_1 > X_2 > \dots >X_\mu>X_{n+1} > X_{\mu+1}>\dots> X_n$. 
Let $p_i^{\opt}(\bX)$ denote the probability that the optimal weak offline solution $\opt$ ends up with agent $i$ holding the item. We now compute these probabilities based on the random arrival order.

For $i\in [\mu]$, the optimal social welfare is $X_i$ (i.e., $\opt = X_i$) if and only if $X_i$ is the highest-valued buyer that arrives after the seller. This requires all better buyers ($X_1, \dots, X_{i-1}$) to arrive before the seller, and $X_i$ to arrive after the seller. Among the set of these $i+1$ agents, the probability of this specific relative ordering is $\frac{1}{i(i+1)}$ (i.e., $p_i^{\opt}(\bX) = \frac{1}{i(i+1)}$). For all buyers $i$ with valuations lower than the seller ($i \in [n]\setminus [\mu]$), they will never be selected in an optimal solution, so $p_i^{\opt}(\bX) = 0$.
The optimal value is $X_{n+1}$ (i.e., $\opt = X_{n+1}$) if and only if no welfare-improving trade is possible. This occurs when the seller arrives after all $\mu$ high-valued buyers. Among these $\mu+1$ agents, the probability of the seller arriving last is $\frac{1}{\mu+1}$ (i.e., $p_{n+1}^{\opt}(\bX) = \frac{1}{\mu+1}$). Therefore, the expected offline optimum is
\begin{align*}
    \opt(\bX)&=\sum_{i=1}^{\mu} p_i^{\opt}(\bX) \cdot X_i + p_{n+1}^{\opt}(\bX) \cdot X_{n+1}\\
    &=\sum_{i=1}^\mu \frac{1}{i(i+1)}\cdot X_i+\frac{1}{\mu+1}\cdot X_{n+1}.
\end{align*}

A key motivation for this separate analysis is the fact that an algorithm optimal under the strong competitive ratio is not necessarily optimal in the weak sense. For instance, the weak competitive ratio of Algorithm~\ref{alg:our} is higher than that of Algorithm~\ref{alg:our2} (presented in Section \ref{sec:weakupper}), even though Algorithm \ref{alg:our} is optimal with respect to the strong competitive ratio.
We illustrate this by considering the following two extreme cases. For the convenience of the subsequent proof, we normalize the values of all $X_i$'s such that $X_i \in [0, 1]$. This operation does not affect the proof itself, but merely simplifies the analysis and description.

Consider the first case where \(X_{n+1} = 1\) and \(X_1 = \dots= X_n = 0\). The offline optimum \(\opt(\bX)\) is clearly 1. For Algorithm~\ref{alg:our}, the social welfare is 1 if and only if the intermediary does not purchase, which happens when the seller arrives after time $(e-1)/e$ (with probability $1/e$). In all other scenarios, the intermediary purchases but cannot resell, yielding a welfare of 0. Thus, the expected algorithmic outcome is \(\talg(\bX) = 1/e\), demonstrating a weak competitive ratio of $e$ for this instance, which is greater than 2.

Consider a second case where the top $k$ buyers have a value of 1 ($X_1=\dots=X_k=1$), while all other agents, including the seller, have a value of 0 ($X_{k+1}=\dots=X_{n+1}=0$). In this setting, the weak offline optimum is 1, unless the seller arrives after all $k$ valuable buyers, in which case the optimum is 0. The probability of the seller arriving last among these $k+1$ agents is $1/(k+1)$, so the expected weak optimum is $\opt(\bX) = \frac{k}{k+1}$, which approaches 1 as $k \to \infty$. For Algorithm~\ref{alg:our}, however, the welfare is 0 if the best buyer, $X_1$, arrives before the seller or within the buyer observation period (i.e., before time $1/e$). This leads to an expected welfare of $\talg(\bX) \le 1/2 - 1/(2e^2)$. Consequently, as $k \to \infty$, the weak competitive ratio is at least $2e^2/(e^2-1)$, which is greater than 2.
 \subsection{Upper bounding the weak ratio}\label{sec:weakupper}

While Algorithm~\ref{alg:our} is optimal for the strong ratio, its performance for the weak ratio is limited. Our analysis indicates that even with optimally tuned thresholds, the weak competitive ratio of Algorithm~\ref{alg:our} cannot be improved beyond $2$. This motivates us to consider a more direct approach, Algorithm~\ref{alg:our2}, which achieves a ratio of $2$ through a simpler randomized mechanism. 

\begin{algorithm}[H]
    \caption{Simple Random Transaction for SPVT}
    \label{alg:our2}

When a seller arrives, purchase the item if the seller's price is not the current highest offer; otherwise, purchase with probability $1/2$. 

When the current highest-offering buyer arrives, sell the item if available. 
\end{algorithm}

\begin{theorem}\label{thm:algweak}
    The weak competitive ratio of Algorithm \ref{alg:our2} is $2$.
\end{theorem}
\begin{proof}

We first consider the case where $X_{n+1} > X_1$. In this case, the optimal solution is to leave the item with the seller, yielding $\opt(\bX) = X_{n+1}$. For our algorithm, since the seller is always the highest bidder upon arrival, the intermediary purchases with probability $1/2$. If the intermediary purchases, the item can never be resold, as no buyer's price is higher than the seller's, resulting in a welfare of $0$. If the intermediary does not purchase, the welfare is $X_{n+1}$. Therefore, the expected welfare is $\talg(\bX) = X_{n+1}/2$. Thus,
\[
\frac{\opt(\bX)}{\talg(\bX)} = \frac{X_{n+1}}{X_{n+1}/2} = 2.
\]

If \( X_1 > X_{n+1} \), let the seller's price be the \((\mu+1)\)-th highest among all agents. Consequently, the \(n-\mu\) buyers with prices lower than the seller's will never be selected and can be ignored in the subsequent analysis.
For any \( i \leq \mu \), the algorithm selects \( X_i \) if and only if the following conditions are satisfied:

\begin{enumerate}
    \item \( X_{n+1} \) arrives before \( X_i \);
    \item \( X_1, \ldots, X_{i-1} \) all arrive after \( X_i \) (this ensures that $X_i$ is the current highest bidder);
    \item Among \( X_{i+1}, \ldots, X_{\mu} \), either:
        \begin{itemize}
            \item at least one buyer arrives before \( X_{n+1} \), or
            \item all arrive after \( X_{n+1} \), but the algorithm still chooses to purchase the item with probability \( 1/2 \) (this ensures the algorithm buys the item from the seller);
        \end{itemize}
    \item Among \( X_{i+1}, \ldots, X_{\mu} \), either:
        \begin{itemize}
            \item all arrive after \( X_i \), or
            \item at least one arrives before \( X_i \), and the highest bidder among them arrives before \( X_{n+1} \) (this ensures the algorithm does not sell the item before \( X_i \) arrives).
        \end{itemize}
\end{enumerate}

Clearly, it is impossible for all \(X_{i+1}, \dots, X_{\mu}\) to arrive after $X_{n+1}$ and before $X_{i}$. Therefore, all the above conditions can be classified into two disjoint cases in total. 

Case 1 is defined as follows:

\begin{enumerate}[label=\arabic*), nosep]
    \item \( X_{n+1} \) arrives before \( X_i \), and \( X_1, \ldots, X_{i-1} \) all arrive after \( X_i \);
    \item Among \( X_{i+1}, \ldots, X_{\mu} \), at least one buyer arrives before \( X_{n+1} \);
    \item Among the buyers in \( X_{i+1}, \ldots, X_{\mu} \) who arrive before \( X_i \), the highest bidder arrives before \( X_{n+1} \).
\end{enumerate}

Let $p_{i1}^{\talg}(\bX)$ denote the probability of obtaining $X_i$ under Case 1. The probability is computed through the following approach: First, calculate the probability that condition 1 is satisfied and at least one buyer among $X_{i+1}, \ldots, X_\mu$ arrives before $X_i$. Then, by performing an exchange between buyers arriving before $X_{n+1}$ and those arriving between $X_{n+1}$ and $X_i$, we can observe that exactly one of the two resulting configurations satisfies the case one. Therefore,
\begin{equation*}
p_{i1}^{\talg}(\bX) = \left(\frac{(i-1)!}{(i+1)!} - \frac{(\mu-1) !}{(\mu+1)!}\right) \cdot \frac{1}{2}=\left(\frac {1 }{i(i+1)} - \frac{1}{\mu(\mu+1)}\right) \cdot \frac{1}{2}.
\end{equation*}

Case 2 is defined as follows:

\begin{enumerate}[label=\arabic*), nosep]
    \item \( X_{n+1} \) arrives before \( X_i \), and \( X_1, \ldots, X_{i-1} \) all arrive after \( X_i \);
    \item  \( X_{i+1}, \ldots, X_{\mu} \) all arrive after \( X_{i} \);
    \item The intermediary decides to purchase the item.
\end{enumerate}

Let $p_{i2}^{\talg}(\bX)$ denote the probability of obtaining $X_i$ under Case 2. The probability that both conditions 1 and 2 are satisfied is $\frac{1}{\mu(\mu+1)} $, and the conditional probability that condition 3 holds given that conditions 1 and 2 are satisfied is $\frac{1}{2}$. Therefore,
\begin{equation*}
p_{i2}^{\talg}(\bX)=\frac{1}{2\mu(\mu+1)}.
\end{equation*}

Hence,
\begin{equation*}
p_{i}^{\talg}(\bX) = p_{i1}^{\talg}(\bX) + p_{i2}^{\talg}(\bX) = \frac{1}{2i(i+1)},  \end{equation*} 
which holds for any \( i \in [\mu] \).

We can determine the probability of $\talg< X_{n+1}$ by analyzing the two possible relative arrival orders of $X_1$ and $X_{n+1}$. First, we consider that $X_1$ arrives before $X_{n+1}$.  When the seller arrives, the highest price seen is $X_1$. Since $X_{n+1} < X_1$, the seller is not the current highest, so the intermediary buys the item. However, because the globally best buyer $X_1$ has already passed, a resale is impossible. The final welfare is $\text{ALG}=0$, satisfying $\talg< X_{n+1}$. Second, we consider that $X_1$ arrives after $X_{n+1}$.  In this case, the final welfare is always at least $X_{n+1}$. This is because the intermediary either refrains from purchasing from the seller (yielding $\text{ALG}=X_{n+1}$) or purchases the item and eventually resells it to a buyer with a value greater than $X_{n+1}$ (such as $X_1$ itself). Thus, the event $\talg< X_{n+1}$ occurs if and only if $X_1$ arrives before $X_{n+1}$. Given the uniformly random arrival order, this occurs with probability $1/2$. Consequently, \( \mathbb{P}(\talg< X_{n+1}) = 1/2\).
Hence,
\begin{align*}
    \frac{\opt(\bX)}{\talg(\bX)} &= \frac{\sum_{i=1}^{\mu} p_{i}^{\opt}(\bX) \cdot X_i + p_{n+1}^{\opt}(\bX) \cdot X_{n+1}}{\sum_{i=1}^\mu p_{i}^{\talg}(\bX) \cdot X_i + p_{n+1}^{\talg}(\bX) \cdot X_{n+1}}\\
    &= \frac{\sum_{i=1}^{\mu} p_{i}^{\opt}(\bX)\cdot (X_i-X_{n+1})+\bp(\opt\ge X_{n+1})\cdot X_{n+1}}{\sum_{i=1}^{\mu} p_{i}^{\talg}(\bX)\cdot (X_i-X_{n+1})+\bp(\talg\ge X_{n+1})\cdot X_{n+1}}\\
    &= \frac{\sum_{i=1}^\mu \frac{1}{i(i+1)}\cdot (X_i-X_{n+1}) + X_{n+1}}{\sum_{i=1}^\mu \frac{1}{2i(i+1)}\cdot (X_i-X_{n+1}) + \frac{1}{2} \cdot X_{n+1}}\\
    &=2. \qedhere
\end{align*}
\end{proof}

\subsection{Double-threshold algorithm for a zero-price seller}
In this subsection, we design a 1.83683-weakly competitive algorithm for the special case of SPVT, denoted as $\text{SPVT}_0$, where the seller's price is fixed at $0$. 
In this case, the intermediary will always buy the item from the seller. This case can be considered a variant of the sample-driven secretary problem, differing from the traditional setting where the sample size is fixed. In our case, the sample size is effectively determined by the seller's uniformly random arrival time.

We now introduce a more sophisticated strategy, the double-threshold algorithm, which adds a second criterion: it may also accept a buyer who is the second-best-so-far buyer. This strategy of accepting a second-ranked bidder is generally counterproductive for the strong competitive ratio, where the benchmark always secures the global maximum. However, it becomes relevant under the weak competitive ratio. The weak offline optimum is constrained by the arrival sequence. Consequently, the value represented by the second rank can be significant relative to this weaker benchmark. Worst-case inputs might involve $X_2$ being very close to $X_1$. Algorithm~\ref{alg:our3} addresses this by potentially accepting the second-best-so-far buyer, but only if they arrive after a later time threshold $t_2$. This allows the algorithm to capture potentially high value from the second rank, particularly late in the process when the chance of encountering and selecting the best candidate has diminished.

\begin{algorithm}[ht]
    \caption{Double-threshold Transaction for $\text{SPVT}_0$} 
    \label{alg:our3}

     Let $T_1, T_2, \ldots, T_{n+1}$ be $n + 1$ independent random variables uniformly distributed on $[0, 1]$. Sort them to obtain the order statistics: $T_{\sigma(1)} < T_{\sigma(2)} < \dots< T_{\sigma(n+1)}$. We refer to $T_{\sigma(i)}$ as the arrival time of the $i$-th participant\;
     When the seller arrives, the intermediary immediately buys the item\;

    After the time the intermediary buys from the seller, let $b$ be the earliest buyer satisfying one of the following conditions   ($t_1$ and $t_2$ are constants that will be determined later)
    \begin{enumerate}
        \item $b$ arrives after $t_1$, and $b$ is the best-so-far buyer;
        \item $b$ arrives after $t_2$, and $b$ is the second-best-so-far buyer.
    \end{enumerate}

    \textbf{If} such a buyer $b$ exists, \textbf{then} the intermediary sells the item to $b$\;
\end{algorithm}

To provide intuition for the performance of Algorithm~\ref{alg:our3}, we outline its asymptotic behavior as the number of buyers $n \to \infty$. Our analysis demonstrates that the worst-case weak competitive ratio is dominated by two extreme price configurations: either $X_1=1$ with all other values being $0$, or $X_1=X_2=\dots=X_n=1$.

In the first extreme case, the expected weak offline optimum is $1/2$. The competitive ratio is determined by the algorithm's probability of successfully stopping at the absolute best buyer. As $n \to \infty$, this probability converges to $(2+t_1^2(3-6t_2)+(3-2t_2)t_2^2+6t_1^2\ln(t_2/t_1))/12$.

In the second extreme case, the optimal benchmark achieves an expected welfare of nearly $1$. The algorithm's performance is thus bounded by the probability that it makes any successful sale. The total probability of making a successful sale converges to $1 - (1/3+t_2^3/6+t_1^2 t_2/2)$.
By balancing these two worst-case scenarios, the asymptotic weak competitive ratio is given by
\[
\min_{0\le t_1\le t_2\le 1}\max\left\{\frac{6}{2+t_1^2(3-6t_2)+(3-2t_2)t_2^2+6t_1^2\ln\frac{t_2}{t_1}},\frac{1}{1 - \left (\frac{1}{3}+\frac{t_2^3}{6}+\frac{t_1^2\cdot t_2}{2}\right )}\right\}
\]
Minimizing this expression yields the optimal threshold parameters $t_1 \approx 0.296151$ and $t_2 \approx 0.805018$, which evaluates to approximately $1.83683$.  Through careful monotonicity analysis (detailed in Appendix~\ref{apx:thm:zeroupper}), we establish that this asymptotic ratio actually serves as a valid upper bound for all finite $n$, leading to the following theorem.

\begin{theorem}\label{thm:zeroupper}
    The weak competitive ratio for Algorithm \ref{alg:our3} is at most $1.83683$ when $t_1=0.296151$ and $t_2=0.805018$.
\end{theorem}

\subsection{Lower bounding the weak ratio} \label{sec:weaklower}
In this subsection, we prove that even when the seller's price is fixed at $0$, the weak competitive ratio for any algorithm is lower-bounded by $1.76239$. 
Note that since this scenario is a valid input for the general SPVT, this bound applies more broadly and improves upon the trivial lower bound of $\frac{2e^2}{e^2+1}\approx 1.76159$, which can be derived from the strong competitive ratio.

Because this subsection adopts the same setting as Section \ref{sec:lower}, where \emph{the seller’s price is fixed at $0$}, Lemma \ref{lma:evo} applies. Following a similar approach to the proof of Theorem~\ref{thm:lower},  we first investigate the optimal 0-value-independent algorithm for two cases: (i) the buyers’ price vector contains a single 1 with all other entries 0, and (ii) it contains exactly two 1s with all other entries 0. We employ again LP tools in our analysis.

\begin{alignat}{4}
    & \max         & \quad & A  & \quad & \label{obj:maxmin}\\
    & \text{ s.t.} &       & j\cdot x_{i,j}\le 1-\sum_{k=i}^{j-1} (x_{i,k}+y_{i,k})  &       &\forall\, 1 \le i \le j \le n \label{9}\\
    &              &       & j\cdot y_{i,j}\le 1-\sum_{k=i}^{j-1} (x_{i,k}+y_{i,k})  &   &\forall\, 1 \le i \le j \le n \label{10}\\
    &              &       & A\le 2 \sum_{i=1}^{n} \sum_{j=i}^{n} \frac{j}{n^2+n} x_{i,j}  &       & \label{11}\\
    &              &       & A\le \frac32 \sum_{i=1}^{n} \sum_{j=i}^{n} \frac{j}{n^2+n} \left(\frac{2n-j}{n} x_{i,j}+\frac{j-1}{n} y_{i,j}\right)  &       &  \label{12}\\
    &              &       & x_{i,j},y_{i,j}\ge 0   &       &  \forall\, 1 \le i \le j \le n  \label{prob:nonneg}
\end{alignat}

\begin{lemma}\label{lma:weak0}
    Given any $0$-value-independent algorithm $\talg$ on $V$ and a price vector $\bX\in V^n$, let $C_i$ be the event that the seller is the $i$-th agent. For any $j \in \{1, \dots, n\}$, let $D_j$ be the event that $\talg$ sells the item to the $j$-th buyer, who is the first best-so-far candidate to appear. Let $E_j$ be the event that $\talg$ sells to the $j$-th buyer, who is the second best-so-far candidate to appear.
    Define $x_{i,j}=\mathbb{P}(D_j|C_i)$ and $y_{i,j}=\mathbb{P}(E_j|C_i)$ for $1\le i\le j\le n$. Let
    \[A = \min\left\{2\cdot p_1^{\talg}(\bX),\, 1.5\cdot \left(p_1^{\talg}(\bX)+p_2^{\talg}(\bX)\right)\right\}.\]
    Then the pair $(x,y,A)$ constitutes a feasible solution to the program \eqref{obj:maxmin}--\eqref{prob:nonneg}. 
\end{lemma}
\begin{proof}
   Let $D'_j$ be the event that $j$-th buyer holds the current highest price. Let $E'_j$ be the event that $j$-th buyer holds the current second-highest price. Let $F_j$ be the event that $X_1$ is the $j$-th buyer to arrive. Let $H_j$ be the event that $X_2$ is the $j$-th to arrive. Let $D''_j$ be the event that the intermediary has an item when the $j$-th buyer arrives. Then for all $i\le j$,
    \[x_{i,j}=\mathbb{P}(D_j|C_i)\le \mathbb{P}(D'_j\cap D''_j|C_i)=\mathbb{P}(D'_j|C_i)\cdot \mathbb{P}(D''_j|C_i)\le \frac{1}{j}\cdot \left(1-\sum_{k=i}^{j-1} (x_{i,k}+y_{i,k})\right),\]
    \[y_{i,j}=\mathbb{P}(E_j|C_i)\le \mathbb{P}(E'_j\cap D''_j|C_i)=\mathbb{P}(E'_j|C_i)\cdot \mathbb{P}(D''_j|C_i)\le \frac{1}{j}\cdot \left(1-\sum_{k=i}^{j-1} (x_{i,k}+y_{i,k})\right),\]
    which implies that $x$ and $y$ satisfies \eqref{9} and\eqref{10}.
    Furthermore, the combination of
    \begin{align*}
       p_1^{\talg}(\bX) = &\sum_{i=1}^n \sum_{j=i}^n \mathbb{P}(C_i \cap D_j \cap F_j) \\
=& \sum_{i=1}^n \sum_{j=i}^n \mathbb{P}(F_j | D_j \cap C_i) \cdot \mathbb{P}(D_j | C_i) \cdot \mathbb{P}(C_i) \\
=& \sum_{i=1}^n \sum_{j=i}^n \mathbb{P}(F_j | D'_j) \cdot \mathbb{P}(D_j | C_i) \cdot \mathbb{P}(C_i) \\
=& \sum_{i=1}^n \sum_{j=i}^n \frac{j}{n} \cdot \frac{1}{n+1} \cdot x_{ij}, \\
 p_2^{\talg}(\bX)=& \sum_{i=1}^n \sum_{j=i}^n \mathbb{P}(C_i \cap D_j \cap H_j) + \mathbb{P}(C_i \cap E_j \cap H_j) \\
=& \sum_{i=1}^n \sum_{j=i}^n \left( \mathbb{P}(H_j | D'_j) \cdot \mathbb{P}(D_j | C_i) + \mathbb{P}(H_j | E'_j) \cdot \mathbb{P}(E_j | C_i) \right) \mathbb{P}(C_i) \\
=& \sum_{i=1}^n \sum_{j=i}^n \left( \frac{\mathbb{P}(H_j \cap D'_j)}{\mathbb{P}(D'_j)} \cdot x_{ij} + \frac{\mathbb{P}(H_j \cap E'_j)}{\mathbb{P}(E'_j)} \cdot y_{ij} \right)\cdot  \frac{1}{n+1} \\
=& \sum_{i=1}^n \sum_{j=i}^n \left( \frac{1}{n} \cdot \frac{n-j}{n}\cdot j \cdot x_{ij} + \frac{1}{n} \cdot \frac{j-1}{n} \cdot j\cdot y_{ij} \right) \cdot \frac{1}{n+1} \\
=& \sum_{i=1}^n \sum_{j=1}^n \left( \frac{j(n-j)}{n^2} \cdot \frac{1}{n+1} \cdot x_{ij} + \frac{j(j-1)}{n^2} \cdot \frac{1}{n+1} \cdot y_{ij} \right)
\end{align*}
 with $A = \min\left\{2\cdot p_1^{\talg}(\bX),\, 1.5\cdot \left(p_1^{\talg}(\bX)+p_2^{\talg}(\bX)\right)\right\}$ yields \eqref{11} and\eqref{12}. 
\end{proof}

\begin{lemma}\label{lma:newlinear}
    The optimal value of the linear program \eqref{obj:maxmin}--\eqref{prob:nonneg} does not exceed $0.567411+o(1)$.
\end{lemma}
\begin{proof}
To provide an upper bound for the optimal value of the linear program \eqref{obj:maxmin}--\eqref{prob:nonneg}, we consider its dual linear program:  

\begin{alignat}{4}
    & \min         & \quad & \sum_{i=1}^{n} \sum_{j=i}^{n} (u_{i,j} + v_{i,j}) & & \label{obj:newdual}\\
    & \text{ s.t.} &       & w_1 + w_2 = 1 & & \\
    &              &       & j \cdot u_{i,j} + \sum_{k=j+1}^{n} (u_{i,k} + v_{i,k}) \ge \frac{2j}{n^2+n}w_1 + \frac{3j(2n-j)}{2n(n^2+n)}w_2 & \quad & \forall\, 1 \le i \le j \le n \label{cons:u}\\
    &              &       & j \cdot v_{i,j} + \sum_{k=j+1}^{n} (u_{i,k} + v_{i,k}) \ge \frac{3j(j-1)}{2n(n^2+n)}w_2 & \quad &  \forall\, 1 \le i \le j \le n \label{cons:v}\\
    &              &       & u_{i,j}, v_{i,j}\ge 0 & \quad & \forall\, 1 \le i \le j \le n\label{cons:nonneg}\\
    &              &       & w_1, w_2 \ge 0 & \quad & \label{prob:newdual}
\end{alignat}

By the duality theory of linear programming, it suffices to prove that the dual program \eqref{obj:newdual}--\eqref{prob:newdual} has a feasible solution $(\mathbf{u}, \mathbf{v},w_1,w_2)$ with an objective value of $0.567411+o(1)$. 

First, we set $w_1 = 0.970659$ and $w_2 = 0.029341$, and determine the values of $\bm{\alpha}$ and $\bm{\beta}$ using Algorithm \ref{alg:alphabeta}.

\begin{algorithm}[ht]
\caption{Algorithm for computing $\bm{\alpha}$ and $\bm{\beta}$}
\label{alg:alphabeta}
$s \gets 0$\;

\For{$j \gets 1$ \KwTo $n$}{$\alpha_j \gets 0$, $\beta_j \gets 0$\;}

\For{$j \gets n$ \KwTo $1$}{
    $\alpha_j \gets \dfrac{\frac{2j}{n^2+n} w_1 + \frac{3j(2n-j)}{2n(n^2+n)} w_2 - s}{j}$, $\beta_j \gets \dfrac{\frac{3j(j-1)}{2n(n^2+n)} w_2 - s}{j}$\;\label{step:ab}

    \If{$\beta_j < 0$}{
        $\beta_j \gets 0$, $j^* \gets j$\;\label{step:j*}
        \textbf{break}\;
    }

    $s \gets s + \alpha_j + \beta_j$\;
}

\For{$j \gets j^*$ \KwTo $1$}{
    $\alpha_j \gets \dfrac{\frac{2j}{n^2+n} w_1 + \frac{3j(2n-j)}{2n(n+n)} w_2 - s}{j}$ \;\label{step:a}

    \If{$\alpha_j < 0$}{
        $\alpha_j \gets 0$, $j^{**} \gets j$ \label{step:j**}\;
        \textbf{break}\;
    }

    $s \gets s + \alpha_j$\;
}
\end{algorithm}

We claim that the constraints \eqref{cons:u} and \eqref{cons:v} are satisfied by letting $u_{i,j} = \alpha_j$ and $v_{i,j}=\beta_j$.  To see this, we just need to observe that during the execution phases of the algorithm, the following properties hold:
\begin{enumerate}
    \item When Steps \ref{step:ab} and \ref{step:a} are executed at iteration $j$, the variable $s$ always holds the value $s=\sum_{k=j+1}^{n} (\alpha_k + \beta_k)$.
    \item The right-hand sides (RHS) of both \eqref{cons:u} and \eqref{cons:v} are monotonically increasing in $j$, and the RHS of \eqref{cons:u} is greater than that of \eqref{cons:v}.
    \item When $\alpha_j$ is assigned in Steps \ref{step:ab} and \ref{step:a} (or $\beta_j$ in Step \ref{step:ab}), constraint \eqref{cons:u} (or \eqref{cons:v}) holds with equality.
    \item When $\alpha_j=0$ (or $\beta_j=0$), it must be that $j \alpha_j=0 \ge \frac{2j}{n^2+n}w_1 + \frac{3j(2n-j)}{2n(n^2+n)}w_2-s$ (or $j \beta_j=0 \ge \frac{3j(j-1)}{2n(n^2+n)}w_2-s$).
\end{enumerate}
Finally, the numerical calculation (for $n>10^6$) yields: \[ \sum_{i=1}^{n} \sum_{j=i}^{n} (u_{i,j} + v_{i,j}) \approx 0.567411,\] 
which implies that the objective value is $0.567411+o(1)$. 
\end{proof}

Lemmas \ref{lma:weak0} and \ref{lma:newlinear} imply the following $ 1.76239-o(1)$ lower bound for the weak competitive ratio, which holds even when the seller's price is fixed at 0. 
\begin{theorem}\label{thm:weaklower} 
    For any algorithm $\talg$ of the SPVT and any $\delta>0$, there exists $n_0$ such that, for all $n\ge n_0$, there exists a price vector ${\bX}'\in\mathbb{N}^n$ satisfying 
    \[\frac{\opt(\bX')}{\talg({\bX}')}\ge 1.76239-\delta.\]
\end{theorem}
\begin{proof}
 By Lemma \ref{lma:evo}, for any $\epsilon>0$ there exists an  algorithm $\talg'$ and infinite set $V\subseteq \mathbb{N}$ such that $\talg'$ is $\epsilon$-value-independent on $V$, and $\talg(\bX)=\talg'(\bX)$ for all $\bX\subseteq V^n$.
 We can define a 0-value-oblivious algorithm $\talg''$ on $V$ by $\mathcal F_{\talg''}=\{r'_{i,j}:1\le i\le j\le n\}$ such that 
 \[r'_{i,j}(x_1,\dots,x_j) = \begin{cases}
q_{i,j,1} & \text{if } x_j \text{ is the largest value in } \{x_1, \dots, x_j\}, \\
q_{i,j,2} & \text{if } x_j \text{ is the second-largest value in } \{x_1, \dots, x_j\}, \\
0 & \text{otherwise}.
\end{cases}
\]
 
 Since $V$ is an infinite subset of $\mathbb{N}$, we can select $n$ distinct numbers $X'_1, X'_2, \dots, X'_n \in V$ such that $X'_1>  X'_2>nX''_3 $ and $X'_3 > X'_4 > \dots > X'_n$. Let $\bX'=(X'_1,X'_2,\ldots,X'_n)$. By Lemmas \ref{lma:weak0} and \ref{lma:newlinear}, we have 
 \[\min\left\{2 p_1^{\talg''}(\bX'),\, 1.5\left(p_1^{\talg''}(\bX')+p_2^{\talg''}(\bX')\right)\right\}\leq 0.567411+o(1).\] 
 Let $\epsilon = 1/n^2$. Since $|p_1^{\talg'}(\bX') - p_1^{\talg''}(\bX')| \le n\epsilon$ and $|p_2^{\talg'}(\bX') - p_2^{\talg''}(\bX')| \le n\epsilon$, it follows that 
 \[\min\left\{2 p_1^{\talg'}(\bX'),\, 1.5 \left(p_1^{\talg'}(\bX')+p_2^{\talg'}(\bX')\right)\right\}\leq 0.567411+o(1).\] 
Then
 \begin{align*}
     \frac{\opt(\bX')}{\talg(\bX')}&\ge \frac{X'_1/2+X'_2/6}{X'_1\cdot p_1^{\talg'}(\bX')+X'_2\cdot p_2^{\talg'}(\bX')+X'_3}\\
     &\ge \max\left\{\frac{1/2}{p_1^{\talg'}(\bX')},\frac{2/3}{p_1^{\talg'}(\bX')+p_2^{\talg'}(\bX')}\right\}-o(1)\\
     &=1.76239-o(1).\qedhere
 \end{align*}
\end{proof}

\begin{remark}
Theorem \ref{thm:weaklower} effectively shows that Algorithm \ref{alg:our3} is asymptotically optimal under the setting where the seller's bid is fixed at 0 and buyer bids consist of either exactly one bid of 1 with all others being 0, or exactly two bids of 1 with all others being 0. To see this, we note that as $n\rightarrow \infty$, the competitive ratio of the algorithm is given by the expression:
\begin{align*}
    &\min_{0\leq t_1 \leq t_2 \leq 1}\max_{\bX}\frac{\opt(\bX)}{\talg(\bX)}\\
    =&\min_{0\leq t_1 \leq t_2 \leq 1}\max_{\bX}\left\{\frac{1}{2p_{1}^{\talg}(\bX)},\frac{\frac{2}{3}}{p_{1}^{\talg}(\bX)+p_{2}^{\talg}(\bX)}\right\}
\end{align*}
where
\begin{align*}
 p_{2}^{\talg}(\bX)&=  \int_0^{t_1} \ds \int_{t_1}^{t_2}(1-t)\cdot \frac{t_1}{t}  \dt 
+ \int_0^{t_1} \ds \int_{t_2}^{1} (1-t)\cdot\frac{t_1}{t} \cdot \frac{t_2}{t}+t_1\cdot \frac{t_2}{t}   \dt \\
&\quad +  \int_{t_1}^{t_2} \ds \int_s^{t_2} (1-t)\cdot \frac{s}{t}  \dt 
+ \int_{t_1}^{t_2} \ds \int_{t_2}^{1} (1-t)\cdot \frac{s}{t} \cdot \frac{t_2}{t}+s\cdot \frac{t_2}{t}  \dt \\
&\quad + \int_{t_2}^{1} \ds \int_s^{1} (1-t)\cdot \frac{s}{t} \cdot \frac{s}{t} +s\cdot \frac{s}{t} \dt\\
&=\frac{1}{12}\left(2+8{t_1}^3+{t_1}^2 (3-12t_2)+(3-4t_2){t_2}^2+6{t_1}^2\ln\frac{t_2}{t_1}\right),
\end{align*}
and 
\[p_{1}^{\talg}(\bX)=\frac{1}{12}\left(2+{t_1}^2(3-6t_2)+(3-2t_2){t_2}^2+6{t_1}^2\ln\frac{t_2}{t_1}\right).\]
When the parameters are set to their optimal values, $t_1 \approx 0.365883$ and $t_2 \approx 0.978772$, Algorithm \ref{alg:our3}  achieves a competitive ratio of $1.76239$.
\end{remark}
\section{Conclusion and future work}
We have established a tight characterization of the strong competitive ratio at $4e^2/(e^2+1) \approx 3.523$, closing the gap for the SPVT problem. The analysis of weak competitive ratios for SPVT opens several promising avenues for future research. We outline key directions that could advance our understanding of online trading problems and their connections to classical optimal stopping problems.

Regarding the weak competitive ratio, we have established a lower bound of 1.76239 and an upper bound of 1.83683 for $\text{SPVT}_0$, and an upper bound of 2 for SPVT. Closing the gaps between the lower and upper bounds remains open. While employing three or more thresholds will naturally yield better competitive ratios, the analytical complexity increases substantially. A compelling theoretical direction is to characterize the underlying structure of the optimal online policy, potentially through a multi-threshold framework, to fully resolve the optimal weak competitive ratio.

Extending the SPVT model to scenarios where the single seller offers multiple homogeneous items, and each buyer is interested in buying one of the items. This can be viewed as a variant of the multiple-choice  secretary problem, introducing new challenges in matching items to buyers while maintaining competitive guarantees. 

Finally, building on the work of Koutsoupias and Lazos \cite{koutsoupias2018online} and Chen et~al. \cite{chen2024algorithms}, investigating settings with multiple sellers could yield practical insights for inventory management and market-making applications.

\clearpage
\appendix

\section{Proof of Lemma \ref{lma:evo}}\label{apx:lma:evo}
The following lemma shows that it suffices to focus on order-oblivious algorithms. 

\begin{lemma}\label{lma:order}
    Given any algorithm $\talg$ for the SPVT, there exists an order-oblivious algorithm $\talg'$ such that $\talg(\bX)=\talg'(\bX)$ for all $\bX$.
\end{lemma}
\begin{proof}
    Let $(x_1,\dots,x_j)_i$ denote the event that the seller is the $i$-th agent (where $i\le j$) and for each  $k\in[j]$, the $k$-th  buyer offers price $x_k$. Let $[x_1]_1=(x_1)_1$, and for all $2\le j\le n$ and $1\le i\le j$, let $[x_1,\dots,x_j]_i$ represent the event  $\cup_{\sigma\in \mathcal{S}_{j-1}} (x_{\sigma(1)},\dots,x_{\sigma(j-1)},x_j)_i$.
We reserve symbol $\tau$ for the decision variable such that algorithm $\talg$ {stops} at the $\tau$-th buyer.
Let algorithm $\talg'$ be defined by $\mathcal F_{\talg'}=\{r'_{i,j}:1\le i\le j\le n\}$ such that 
    \[r'_{i,j}(x_1,\dots,x_j)=\begin{cases}
        \mathbb{P}(\tau=j\,|\,(\tau\ge j)\cap [x_1,\dots,x_j]_i) & \text{if $\mathbb{P}(\tau\ge j\,|\,[x_1,\dots,x_j]_i)>0$,}\\
        0 & \text{otherwise.}
    \end{cases}\]
    It is clear that $\talg'$ is order-oblivious. Suppose that algorithm $\talg'$ stops at the $\tau'$-th agent. Next we show by induction on $j$ that for all $1\le i\le j\le n$ and $(x_1,\dots,x_j)\in \mathbb{R}_+^j$,
\begin{equation}\label{eq:order}
    \mathbb{P}(\tau=j\,|\,[x_1,\dots,x_j]_i)=\mathbb{P}(\tau'=j\,|\,[x_1,\dots,x_j]_i).
\end{equation}
First, we show that (\ref{eq:order}) holds when   $i=j$, which particularly implies the validity of (\ref{eq:order}) for $j=1$ (the base case). Indeed, $\mathbb{P}(\tau\ge i\,|\,[x_1,\dots,x_i]_i)=1$ and 
    \begin{align*}
        \mathbb{P}(\tau'=i\,|\,[x_1,\dots,x_i]_i)&=\frac{1}{(i-1)!}\sum_{\sigma\in \mathcal{S}_{i-1}} r'_{i,i}(x_{\sigma(1)},\dots,x_{\sigma(i-1)},x_i)\\
        &=\frac{1}{(i-1)!}\sum_{\sigma\in \mathcal{S}_{i-1}} \mathbb{P}(\tau=i\,|\,[x_1,\dots,x_i]_i)\\
        &=\mathbb{P}(\tau=i\,|\,[x_1,\dots,x_i]_i).
    \end{align*}
  Proceeding inductively, we assume that $j\ge2$ and $\mathbb{P}(\tau=k\,|\,[x_1,\dots,x_k]_i)=\mathbb{P}(\tau'=k\,|\,[x_1,\dots,x_k]_i)$ holds for all  $1\le i\le k\le j-1$. Because
    \begin{align*}
        \mathbb{P}(\tau=k\,|\,[x_1,\dots,x_j]_i)&=\frac{1}{(j-1)!}\sum_{\sigma\in \mathcal{S}_{j-1}}\mathbb{P}(\tau=k\,|\,(x_{\sigma(1)},\dots,x_{\sigma(k)})_i)\\
        &=\frac{1}{(j-1)!}\sum_{\sigma\in \mathcal{S}_{j-1}}\frac{1}{(k-1)!}\sum_{\pi\in \mathcal{S}_{k-1}} \mathbb{P}(\tau=k\,|\,(x_{\pi(\sigma(1))},\dots,x_{\pi(\sigma(k-1))},x_{\sigma(k)})_i)\\
        &=\frac{1}{(j-1)!}\sum_{\sigma\in \mathcal{S}_{j-1}}\mathbb{P}(\tau=k\,|\,[x_{\sigma(1)},\dots,x_{\sigma(k)}]_i)\\
        &=\frac{1}{(j-1)!}\sum_{\sigma\in \mathcal{S}_{j-1}}\mathbb{P}(\tau'=k\,|\,[x_{\sigma(1)},\dots,x_{\sigma(k)}]_i)\\
        &=\mathbb{P}(\tau'=k\,|\,[x_1,\dots,x_j]_i),
    \end{align*}
    we have
    \begin{align*}
        \mathbb{P}(\tau\ge j\,|\,[x_1,\dots,x_j]_i)&=1-\sum_{k=i}^{j-1}\mathbb{P}(\tau=k\,|\,[x_1,\dots,x_j]_i)\\
        &=1-\sum_{k=i}^{j-1}\mathbb{P}(\tau'=k\,|\,[x_1,\dots,x_j]_i)\\
        &=\mathbb{P}(\tau'\ge j\,|\,[x_1,\dots,x_j]_i).
    \end{align*}
    If $\mathbb{P}(\tau\ge j\,|\,[x_1,\dots,x_j]_i)=0$, then (\ref{eq:order}) is trivial. Otherwise, it leads that
    \begin{align*}
        \mathbb{P}(\tau'=j|[x_1,\dots,x_j]_i)&=\frac{1}{(j-1)!}\sum_{\sigma\in \mathcal{S}_{j-1}} r'_{i,j}(x_{\sigma(1)},\dots,x_{\sigma(j-1)},x_j)\cdot \mathbb{P}(\tau'\ge j|(x_1,\dots,x_j)_i)\\
        &=r'_{i,j}(x_1,\dots,x_j)\cdot \mathbb{P}(\tau'\ge j|[x_1,\dots,x_j]_i)\\
        &=\frac{\mathbb{P}(\tau=j|[x_1,\dots,x_j]_i)}{\mathbb{P}(\tau\ge j|[x_1,\dots,x_j]_i)}\cdot \mathbb{P}(\tau\ge j|[x_1,\dots,x_j]_i)\\
        &=\mathbb{P}(\tau=j|[x_1,\dots,x_j]_i).
    \end{align*}
    Thus, we have
    \begin{align*}
        \talg(\bX)&=\frac{1}{(n+1)!}\sum_{\sigma\in \mathcal{S}_n} \sum_{i=1}^{n} \sum_{j=i}^{n} X_{\sigma(j)} \cdot \mathbb{P}(\tau=j|(X_{\sigma(1)},\dots,X_{\sigma(j)})_i)\\
        &=\frac{1}{(n+1)!}\sum_{\sigma\in \mathcal{S}_n} \sum_{i=1}^{n} \sum_{j=i}^{n} \frac{X_{\sigma(j)}}{(j-1)!}\sum_{\pi\in \mathcal{S}_{j-1}} \mathbb{P}(\tau=j|(X_{\pi(\sigma(1))},\dots,X_{\pi(\sigma(j-1))},X_{\sigma(j)})_i)\\
        &=\frac{1}{(n+1)!}\sum_{\sigma\in \mathcal{S}_n} \sum_{i=1}^{n} \sum_{j=i}^{n} X_{\sigma(j)} \mathbb{P}(\tau=j|[X_{\sigma(1)},\dots,X_{\sigma(j)}]_i)\\
        &=\frac{1}{(n+1)!}\sum_{\sigma\in \mathcal{S}_n} \sum_{i=1}^{n} \sum_{j=i}^{n} X_{\sigma(j)} \mathbb{P}(\tau'=j|[X_{\sigma(1)},\dots,X_{\sigma(j)}]_i)\\
        &=\frac{1}{(n+1)!}\sum_{\sigma\in \mathcal{S}_n} \sum_{i=1}^{n} \sum_{j=i}^{n} X_{\sigma(j)} \mathbb{P}(\tau'=j|(X_{\sigma(1)},\dots,X_{\sigma(j)})_i)\\
        &=\talg'(\bX). \qedhere
    \end{align*}
\end{proof}

We utilize the following Ramsey theorem on infinite hypergraphs as a tool to narrow down the price vector space that needs to be considered.

\begin{lemma}[Ramsey~\cite{ramsey1930problem}]\label{lma:ramsey}
    Let $c,d\in \mathbb{N}$, and let $H$ be an infinite complete d-uniform hypergraph whose hyperedges are colored with $c$ colors. Then there exists an infinite complete $d$-uniform subhypergraph of $H$ that is monochromatic.
\end{lemma}

Now, we are ready to prove Lemma \ref{lma:evo}.
\begin{proof}[Proof of Lemma \ref{lma:evo}]
    By Lemma \ref{lma:order}, there is an order-oblivious algorithm $\talg'$ such that $\talg(\bX) = \talg'(\bX)$ for any $\bX \subseteq \mathbb{N}^n$.

    Next, we iteratively construct the corresponding set $V$ on which $\talg'$ is $\epsilon$-value-independent. Initially, we set $V = \mathbb{N}$. Then, we conduct a double loop, where the outer loop grows $i$   from 1 to $n$  and the inner loop grows $j$  from $i$ to $n$, to gradually narrow the set $V$ to make the algorithm $\talg$   $(\epsilon,i,j)$-value-independent on it for all $1\le i\le j\le n$. 

   Specifically, given $i$ and $j$, let $H$ denote the complete $j$-uniform hypergraph on the vertex set $V$. For each hyeredge $\{v_1, \dots, v_j\} \in V^j$ with $v_1>v_2\dots>v_j$, because  $r_{i,j}\in [0, 1]$ and $\talg'$ is order-oblivious, there is $(u_1,\dots,u_j)$ with $u_k \in \{1, 2, \dots, \lceil 1/(2\epsilon) \rceil \}$ such that $r_{i,j}(v_{\sigma(1)}, \dots, v_{\sigma(j)}) \in [(2u_{\sigma(j)}-1) \cdot \epsilon - \epsilon, (2u_{\sigma(j)}-1) \cdot \epsilon + \epsilon)$; we color this hyperedge with color $u$. Now we apply Lemma \ref{lma:ramsey} with $c = \lceil 1/(2\epsilon) \rceil$ and $d = j$, and obtain an infinite subhypergraph $H'$ of $H$ that is monochromatic. This monochromaticity ensures that the algorithm $\talg$ is $(\epsilon, i, j)$-value-independent on the vertex set of $H'$. We narrow $V$ to be the vertex set of $H'$, and continue to the narrowing process for the next $i,j$ (defined by the double loop). Finally, we end up with an infinite set $V$ as desired. 
\end{proof}

\section{Proof of Theorem \ref{thm:zeroupper}}\label{apx:thm:zeroupper}
We begin by analyzing the probability that the item is ultimately allocated to agent $X_i$, which is denoted by $p_{i}^{\talg}$. Note that since our algorithm depends only on the relative ranks of the buyers, the value of $p_{i}^{\talg}$ is independent of the values in $\bX$. For clarity of exposition, let $s$ and $t$ denote the arrival times of the seller $X_{n+1}$ and the buyer $X_i$, respectively.

Let $B_{l,m,t}=\{i:l\le i\le m$,  buyer $i$  arrives before $t\}$ be the set of buyers who offer a higher price than $X_m$ and a lower price than $X_l$ and arrive earlier than time $t$. When $B_{l,m,t}\neq \emptyset$, for any $r\in[0,t]$, let $A_{r,l,m,t}$ denote the event that the best buyer (the one with the highest price) in $B_{l,m,t}$ arrives before time $r$, and $A'_{r,l,m,t}$ denote the event that the second-best buyer (the one with the second-highest price) in $B_{l,m,t}$ arrives before time $r$.
      
If buyer $X_i$ obtains the item, then the seller $X_{n+1}$ must arrive before $X_i$, and $X_i$ must arrive after time $t_1$. Furthermore, one of the following two conditions must hold: 

Case 1 is defined as follows:

\begin{enumerate}[label=\arabic*), nosep]
    \item All buyers $X_1, \ldots, X_{i-1}$ arrive after time $t$ (this ensures that $X_i$ is the current highest bidder);
    \item  Among $X_{i+1},\ldots, X_n$, either:
       \begin{itemize}
            \item  all arrive after time $t$, or exactly one buyer among them arrives before $\max\{s, t_1\}$, or
            \item at least two buyers arrive before time $t$, with the highest bidder among them arriving before $\max\{s, t_1\}$ and the second-highest bidder arriving before $\max\{s, t_2\}$.\\
             (These ensure that $X_i$ is the first best-so-far bidder after time $t_1$)
       \end{itemize}
\end{enumerate}

Case 2 is defined as follows:
\begin{enumerate}[label=\arabic*), nosep]
    \item Among $X_1, \ldots, X_{i-1}$, exactly one buyer arrives before $\max\{s, t_1\}$, while all remaining ones arrive after time $t$ 
    \item All buyers $X_{i+1}, \ldots, X_n$ arrive after time $t$; or there exists at least one buyer among them arriving before time $t$, with the highest bidder among these early arrivals arriving before $\max\{s, t_2\}$. 
    (These ensure that $X_i$ is the first second-best-so-far bidder after time $t_2$);
    \item $t \geq t_2$.
\end{enumerate}
\medskip

Let $p_{i1}^{\talg}$ be the probability of obtaining $X_i$ under Case $1$. Since the relationship between $s$ and the critical values $t_1$ and $t_2$ is variable, the integral must be analyzed by dividing the domain into the subintervals $0 \le s \le t_1$, $t_1 \le s \le t_2$, and $t_2 \le s \le 1$. We have
\begin{align*}
     p_{i1}^{\talg}&=
     \int_{0}^{t_1}\ds \int_{t_1}^{1} \bp(B_{1,i-1,t}=\emptyset)\cdot \left ( \bp(B_{i+1,n,t}=\emptyset) + \bp( A_{t_1,i+1,n,t} | |B_{i+1,n,t}|=1)\right ) \dt\\
     &\quad+\int_{0}^{t_1}\ds \int_{t_1}^{t_2} \bp(B_{1,i-1,t}=\emptyset)\cdot \bp( A_{t_1,i+1,n,t} | |B_{i+1,n,t}|\ge 2) \dt\\
     &\quad+\int_{0}^{t_1}\ds \int_{t_2}^{1} \bp(B_{1,i-1,t}=\emptyset)\cdot \bp( A_{t_1,i+1,n,t}\cap A'_{t_2,i+1,n,t} | |B_{i+1,n,t}|\ge 2) \dt\\
     &\quad+\int_{t_1}^{t_2}\ds \int_{s}^{1} \bp(B_{1,i-1,t}=\emptyset)\cdot \left ( \bp(B_{i+1,n,t}=\emptyset) + \bp( A_{s,i+1,n,t} | |B_{i+1,n,t}|=1)\right ) \dt\\
     &\quad+\int_{t_1}^{t_2}\ds \int_{s}^{t_2} \bp(B_{1,i-1,t}=\emptyset)\cdot \bp( A_{s,i+1,n,t} | |B_{i+1,n,t}|\ge 2) \dt\\
     &\quad+\int_{t_1}^{t_2}\ds \int_{t_2}^{1} \bp(B_{1,i-1,t}=\emptyset)\cdot \bp( A_{s,i+1,n,t}\cap A'_{t_2,i+1,n,t} | |B_{i+1,n,t}|\ge 2) \dt\\
     &\quad +\int_{t_2}^{1}\ds \int_{s}^{1} \bp(B_{1,i-1,t}=\emptyset)\cdot \left ( \bp(B_{i+1,n,t}=\emptyset) + \bp( A_{s,i+1,n,t} | |B_{i+1,n,t}|=1)\right ) \dt\\
     &\quad+\int_{t_2}^{1}\ds \int_{s}^{1} \bp(B_{1,i-1,t}=\emptyset)\cdot \bp( A_{s,i+1,n,t}\cap A'_{s,i+1,n,t} | |B_{i+1,n,t}|\ge 2) \dt.
\end{align*}
Note that
\begin{align*}
    \bp(B_{1,i-1,t}=\emptyset)&=(1-t)^{i-1},\\  
    \bp(B_{i+1,n,t}=\emptyset)&=(1-t)^{n-i},\\
    \bp( A_{r,i+1,n,t} | |B_{i+1,n,t}|=1)&=r(n-i)(1-t)^{n-1-i},\\
    \bp( A_{r,i+1,n,t} | |B_{i+1,n,t}|\ge 2)&=\left(1-(1-t)^{n-i} -(n-i) t (1-t)^{n-1-i}\right) \cdot \frac{r}{t},\\
    \bp( A_{r_1,i+1,n,t}\cap A'_{r_2,i+1,n,t} | |B_{i+1,n,t}|\ge 2)&=\left(1-(1-t)^{n-i} -(n-i) t (1-t)^{n-1-i}\right) \cdot\frac{r_1r_2}{t^2}.
\end{align*}
Thus,
\begin{equation}\label{eq:pi1alg}
\begin{aligned}
p_{i1}^{\talg}&= \int_0^{t_1} \ds \int_{t_1}^{1} (1-t)^{n-1}+t_1(n-i)(1-t)^{n-2}    \dt\\
&\quad + \int_0^{t_1} \ds \int_{t_1}^{t_2} \left ((1-t)^{i-1}-(1-t)^{n-1} -t(n-i)(1-t)^{n-2}\right ) \cdot \frac{t_1}{t}  \dt
\\
&\quad+\int_0^{t_1} \ds \int_{t_2}^{1} \left ((1-t)^{i-1}-(1-t)^{n-1} -t(n-i) (1-t)^{n-2}\right ) \cdot \frac{t_1 t_2}{t^2}  \dt
\\
&\quad+ \int_{t_1}^{1} \ds \int_{s}^{1} (1-t)^{n-1}+s(n-i) (1-t)^{n-2}   \dt\\
&\quad +\int_{t_1}^{t_2} \ds \int_{s}^{t_2} \left ((1-t)^{i-1}-(1-t)^{n-1} -t(n-i) (1-t)^{n-2}\right ) \cdot \frac{s}{t}    \dt
\\
&\quad +\int_{t_1}^{t_2} \ds \int_{t_2}^{1} \left ((1-t)^{i-1}-(1-t)^{n-1} -t(n-i) (1-t)^{n-2}\right ) \cdot  \frac{st_2}{t^2} \dt
\\
&\quad+ \int_{t_2}^{1} \ds \int_{s}^{1} (1-t)^{n-1}+s(n-i) (1-t)^{n-2}   \dt\\
&\quad +\int_{t_2}^{1} \ds \int_{s}^{1} \left ((1-t)^{i-1}-(1-t)^{n-1} -t(n-i) (1-t)^{n-2}\right )\cdot \frac{s^2}{t^2}  \dt.
\end{aligned}
\end{equation}

Let $p_{i2}^{\talg}$ be the probability of obtaining $X_i$ in Case $2$. Note that $p_{12}^{\talg}=0$. Since the relationship between $s$ and $t_1$ and $t_2$ is undetermined, the integral must be analyzed by dividing the domain into the subintervals $0 \le s \le t_1$, $t_1 \le s \le t_2$, and $t_2 \le s \le 1$. It is clear that
\begin{align*}
     p_{i2}^{\talg}&=
     \int_0^{t_1} \ds \int_{t_2}^{1} \bp( A_{t_1,1,i-1,t} | |B_{1,i-1,t}|=1)\cdot (\bp(B_{i+1,n,t}=\emptyset)+\bp( A_{t_2,i+1,n,t} | B_{i+1,n,t}\neq \emptyset))\dt\\
     &\quad +\int_{t_1}^{t_2} \ds \int_{t_2}^{1} \bp( A_{s,1,i-1,t} | |B_{1,i-1,t}|=1)\cdot (\bp(B_{i+1,n,t}=\emptyset)+\bp( A_{t_2,i+1,n,t} | B_{i+1,n,t}\neq \emptyset))\dt\\
     &\quad +\int_{t_2}^{1} \ds \int_{s}^{1} \bp( A_{s,1,i-1,t} | |B_{1,i-1,t}|=1)\cdot (\bp(B_{i+1,n,t}=\emptyset)+\bp( A_{s,i+1,n,t} | B_{i+1,n,t}\neq \emptyset))\dt.
\end{align*}
Note that
\begin{align*}
    \bp( A_{r,1,i-1,t} | |B_{1,i-1,t}|=1)&=r(i-1)(1-t)^{i-2},\\
    \bp(B_{i+1,n,t}=\emptyset)&=(1-t)^{i-2},\\
    \bp( A_{r,i+1,n,t} | B_{i+1,n,t}\neq \emptyset)&=(1-(1-t)^{n-i}) \cdot \frac{r}{t}.
\end{align*}
Thus, for any $i\ge 2$,
\begin{equation}\label{eq:pi2alg}
    \begin{aligned}
p_{i2}^{\talg} =& \int_0^{t_1} \ds \int_{t_2}^{1} (i-1)\cdot t_1\cdot (1-t) ^{i-2}\cdot\left ((1-t)^{n-i}+\left (1-(1-t)^{n-i}\right )\cdot \frac{t_2}{t}\right )  \dt
\\
&+\int_{t_1}^{t_2} \ds \int_{t_2}^{1} (i-1)\cdot s\cdot (1-t) ^{i-2}\cdot\left ((1-t)^{n-i}+\left (1-(1-t)^{n-i}\right )\cdot \frac{t_2}{t}\right )  \dt
\\
&+\int_{t_2}^{1} \ds \int_{s}^{1} (i-1)\cdot s\cdot (1-t) ^{i-2}\cdot\left ((1-t)^{n-i}+\left (1-(1-t)^{n-i}\right )\cdot \frac{s}{t}\right )  \dt.
\end{aligned}
\end{equation}


To analyze the monotonicity of $p_i^{\talg}$, we next introduce Lemma~\ref{lm:mono}, which will serve as necessary tools.

Let $a, b \in [0, 1]$ with $a < b$. 
Let $A, B: [a, b] \to [0, 1]$ be continuous functions such that $0 < \alpha(s) < \beta(s) \le 1$ for all $s \in [a, b]$. 
Let $f: [0, 1]^2 \to \mathbb{R}_+$ be a continuous function. 
For integers $i \ge 1$, define the sequences $G(i)$ and $H(i)$ as:
\[G(i)=i(i+1) \int_{a}^{b} \ds\int_{\alpha(s)}^{\beta(s)} f(t, s) (1-t)^{i-1}\dt,\]
\[H(i)=(i-1)i(i+1)\int_{a}^{b} \ds\int_{\alpha(s)}^{\beta(s)} f(t, s) (1-t)^{i-2}\dt.\]
For integers $i \ge 1$, define the first difference sequences:
    \[ \Delta G(i) = G(i) - G(i+1), \quad  \Delta H(i) = H(i) - H(i+1).\]
The following lemma states that for sufficiently large $i$, both $\Delta G(i)$ and $\Delta H(i)$ are decreasing.
\begin{lemma}\label{lm:mono}
    Let $t_* = \min_{s \in [a, b]}\alpha(s)$, and
    \begin{align*}
        I_1&=\left\lceil\frac{4-5t_*+\sqrt{t_*^2-8t_*+8}}{2t_*}\right\rceil\\
        I_2&=\left\lceil\frac{6 - 5t_* + \sqrt{t_*^2 - 12t_* + 12}}{2t_*}\right\rceil.
    \end{align*}
    Then for any integer $i \ge I_1$, the sequence $\Delta G(i)$ is decreasing, and for any integer $i \ge I_2$, the sequence $\Delta H(i)$ is decreasing.
\end{lemma}
\begin{proof}
    First, we analyze $\Delta G(i)$. The second difference is
\begin{align*}
    \Delta G(i+1) - \Delta G(i) = \int_{a}^{b}\ds\int_{\alpha(s)}^{\beta(s)} f(t,s) (1-t)^{i-1} \cdot p(t, i) \dt
\end{align*}
where
\begin{align*}
    p(t, i) &= -(i+2)(i+3)(1-t)^2+2(i+1)(i+2)(1-t)-i(i+1)\\
    &= -t^2 i^2 + (4t-5t^2)i -6t^2+8t-2.
\end{align*}
For a fixed $t$, the roots of the equation $p(t, i)=0$ are:
\[ P_1(t) = \frac{4-5t-\sqrt{t^2-8t+8}}{2t}, \quad P_2(t) =  \frac{4-5t+\sqrt{t^2-8t+8}}{2t}.\]
It can be easily verified that when $t \in (0, 1]$, we have $P_1(t) < P_2(t)$ and $P_2(t)$ is monotonically decreasing. Therefore, for any $i \ge I_1 = \lceil P_2(t_*)\rceil$ and $t\in [\alpha(s),\beta(s)]$, we have $p(t,i) \le 0$, which implies that the sequence $\Delta G(i)$ is monotonically decreasing.

Then, we analyze $\Delta H(i)$. The second difference is
\begin{align*}
\Delta H(i+1) - \Delta H(i) = \int_{a}^{b} \ds \int_{\alpha(s)}^{\beta(s)} f(t,s) (1-t)^{i-2} (i+1) \cdot q(t,i) \dt
\end{align*}
where
\begin{align*}
    q(t, i)&=-(i+2)(i+3)(1-t)^2+2i(i+2)(1-t)-i(i-1)\\
    &=-t^2i^2 + (6t-5t^2)i-6(1-t)^2
\end{align*}
For a fixed $t$, the roots of the equation $q(t, i)=0$ are:
\[ Q_1(t) = \frac{6 - 5t - \sqrt{t^2 - 12t + 12}}{2t}, \quad Q_2(t) =  \frac{6 - 5t + \sqrt{t^2 - 12t + 12}}{2t}.\]
It can be easily verified that when $t \in (0, 1]$, we have $Q_1(t) < Q_2(t)$ and $Q_2(t)$ is monotonically decreasing. Therefore, for any $i \ge I_2 = \lceil Q_2(t_*)\rceil$ and $t\in [\alpha(s),\beta(s)]$, we have $q(t,i) \le 0$, which implies that the sequence $\Delta H(i)$ is monotonically decreasing. 
\end{proof}

\begin{lemma}\label{lma 4.4}
 Let $t_1 = 0.296151$ and $t_2 = 0.805018$ in Algorithm 3. Then the competitive ratio of Algorithm 3 is
 \[\max_{\bX}\frac{\opt(\bX)}{\talg(\bX)}=\max\left\{\frac{1}{2p_1^{\talg}},\frac{n}{(n+1)\sum_{i=1}^n p_i^{\talg}}\right\}.\]
\end{lemma}

\begin{proof}
The following ratio
\[\frac{\opt(\bX)}{\talg(\bX)} 
=\frac{\sum_{i=1}^n \frac{1}{i(i+1)}\cdot X_i }{\sum_{i=1}^n p_i^{\talg} \cdot X_i}\]
is a ratio of two linear polynomials in the variables $X_i$. 
Let $f(i, n)=i  (i+1)  \cdot p_i^{\talg}$. Then by Equations \eqref{eq:pi1alg} and \eqref{eq:pi2alg}, we have
\[f(i,n)=\, i (i+1)\cdot p_i^{\talg}=\sum_{k=1}^8 \alpha_k(i,n)=\sum_{k=1}^8(\beta_{k1}(i)+\beta_{k2}(i,n)),\]
where $\alpha_k(i,n)=\beta_{k1}(i)+\beta_{k2}(i,n)$ and 
\begin{align*}
    \beta_{11}(i)&=\int_{0}^{t_1} \ds\int_{t_1}^{t_2} \frac{t_1}{t}  (1-t)^{i-1}  i (i+1) \dt \\\
    \beta_{12}(i,n)&=\int_{0}^{t_1} \ds\int_{t_1}^{t_2} \left(1-\frac{t_1}{t}\right) (1-t)^{n-1}  i (i+1) \dt \\
    \beta_{21}(i)&=\int_{t_1}^{t_2} \ds\int_{s}^{t_2}\frac{s}{t}  (1-t)^{i-1}  i (i+1)  \dt\\
    \beta_{22}(i,n)&=\int_{t_1}^{t_2} \ds\int_{s}^{t_2}\left(1-\frac{s}{t}\right) (1-t)^{n-1}  i (i+1)  \dt\\
    \beta_{31}(i)&=\int_{0}^{t_1}\ds \int_{t_2}^{1} \frac{t_1t_2}{t^2} \ (1-t)^{i-1}  i (i+1)  \dt\\
    \beta_{32}(i,n)&=\int_{0}^{t_1}\ds \int_{t_2}^{1} \left(1-\frac{t_1t_2}{t^2}\right)\ (1-t)^{n-1}  i (i+1)  \dt\\
    \beta_{41}(i)&=\int_{t_1}^{t_2}\ds \int_{t_2}^{1} \frac{st_2}{t^2}  (1-t)^{i-1}  i (i+1)  \dt\\
    \beta_{42}(i,n)&=\int_{t_1}^{t_2}\ds \int_{t_2}^{1} \left(1-\frac{st_2}{t^2}\right) (1-t)^{n-1}  i (i+1)  \dt\\
    \beta_{51}(i)&=\int_{t_2}^{1}\ds\int_{s}^{1} \frac{s^2}{t^2}  (1-t)^{i-1}  i (i+1)\dt\\
    \beta_{52}(i,n)&=\int_{t_2}^{1}\ds\int_{s}^{1} \left(1-\frac{s^2}{t^2}\right) (1-t)^{n-1}  i (i+1)\dt\\
    \beta_{61}(i)&=\int_{0}^{t_1}\ds \int_{t_2}^{1}\frac{t_1t_2}{t} (1-t)^{i-2} (i-1)i(i+1)  \dt\\
    \beta_{62}(i,n)&=\int_{0}^{t_1}\ds \int_{t_2}^{1} t_1 \left(1-\frac{t_2}{t}  (n-1) (1-t)^{n-2}\right)  i (i+1)  \dt\\
    \beta_{71}(i)&=\int_{t_1}^{t_2}\ds \int_{t_2}^{1} \frac{st_2}{t} (1-t)^{i-2} (i-1)i(i+1)   \dt\\
    \beta_{72}(i,n)&=\int_{t_1}^{t_2}\ds \int_{t_2}^{1} s \left(1-\frac{t_2}{ t}   (n-1) (1-t)^{n-2}\right)  i (i+1)   \dt\\
    \beta_{81}(i)&=\int_{t_2}^{1}\ds\int_{s}^{1}  \frac{s^2}{t} (1-t)^{i-2}(i-1)i(i+1)   \dt\\
    \beta_{82}(i,n)&=\int_{t_2}^{1}\ds\int_{s}^{1}  s \left(1-\frac{s}{ t}\right)   (n-1) (1-t)^{n-2} i (i+1)   \dt.
\end{align*}

Next, we will show that $f(i, n)$ increases as $i$ goes from $1$ to $2$, and then decreases as $i$ goes from $2$ to $n$,
which is instrumental in proving that the maximum value of the competitive ratio, $\max_{\bX}\frac{\opt(\bX)}{\talg(\bX)}$, is attained only in the two possible cases: $\bX=(1,0,\dots,0)$ or $\bX=(1,1,\dots,1)$.
We divide the entire proof into three parts:

Part 1: To prove that there exist thresholds $N_0=10$ and $I_0=3$ such that for all $n > N_0$ and $i \ge I_0$, $f(i, n)$ is monotonically decreasing with respect to $i$.

Part 2: To prove that for all $n\in \mathbb{N}$, $f(1,n)<f(2,n)$ and $f(2,n)>f(3,n)$.

Part 3: To prove that for all $n \le N_0$, $f(i, n)$ increases as $i$ goes from $1$ to $2$, and then decreases as $i$ goes from $2$ to $n$.

\paragraph{Part 1:} For notational convenience, we define $M_{k1}(i) = \beta_{k1}(i) - \beta_{k1}(i+1)$ and $M_{k2}(i, n) = \beta_{k2}(i+1, n) - \beta_{k2}(i, n)$. 
Since
\begin{align*}
    \sum_{k=1}^8 M_{k1}(i)-M_{k2}(i,n)&=\sum_{k=1}^8\beta_{k1}(i) - \beta_{k1}(i+1)-(\beta_{k2}(i+1, n) - \beta_{k2}(i, n))\\
    &=\sum_{k=1}^8\beta_{k1}(i)+\beta_{k2}(i, n) - \beta_{k1}(i+1)-\beta_{k2}(i+1, n)\\
    &=f(i,n)-f(i+1,n),
\end{align*}
we only need to prove that for all $k\in \{1,\dots,8\}$, there exist $N_k$ and $I_k$ such that for all $n > N_k$ and $I_k \le i<n$, we have $M_{k1}(i)> M_{k2}(i,n)$.

We now focus on $\alpha_1$. By Lemma \ref{lm:mono}, 
there exists a threshold $I_1' = \left\lceil\frac{4-5t_1+\sqrt{t_1^2-8t_1+8}}{2t_1}\right\rceil=9$, such that $M_{11}(i)$ is monotonically decreasing with respect to $i$ for all $i \ge I_1'$.  
Furthermore, numerical computation shows that $M_{11}(i) \ge M_{11}(9)$ for all $i=3,\dots,9$.
And when $n > \lceil \frac{2}{t_{1}} \rceil - 1=6$, we have 
\begin{align*}
M_{11}(n-1)&=\int_{0}^{t_1} \ds\int_{t_1}^{t_2} \frac{t_1}{t}\left ((n-1)n(1-t)^{n-2}-n(n+1)(1-t)^{n-1}\right )\dt \\&>\int_{0}^{t_1} \ds\int_{t_1}^{t_2} 2n\left (1-\frac{t_1}{t}\right )(1-t)^{n-1}\dt\\&= M_{12}(n-1, n).
\end{align*}
Thus, by setting $N_1 = \max \{ 10, 6 \} = 10$ and $I_1 = 3$, for all $n > N_1$ and $I_1 \le i<n$, we have
\[M_{11}(i)\ge M_{11}(n-1)>M_{12}(n-1,n)\ge M_{12}(i,n).\]

Next, we consider $ \alpha_2$.
By Lemma \ref{lm:mono}, there exists a threshold $I_2' = \left\lceil\frac{4-5t_1+\sqrt{t_1^2-8t_1+8}}{2t_1}\right\rceil=9$, such that $M_{21}(i)$ is monotonically decreasing with respect to $i$ for all $i \ge I_2'$.  
Supplementing this with numerical computations for $i=1, \dots, 9$, we find that $M_{21}(i)$ actually decreases from $3$, which means that for all $3 \le i \le n-1$, $M_{21}(i) \ge M_{21}(n-1)$. 
And when $n>\max _{t\in [s,t_2],s\in [t_1,t_2]} \lceil \frac{2}{s}- \frac{2t}{s} \rceil + 1= \lceil \frac{2}{t_{1}} \rceil - 1=6$, we have
\begin{align*}
M_{21}(n-1)&=\int_{t_1}^{t_2} \ds\int_{s}^{t_2}\frac{s}{t}\left ((n-1)n(1-t)^{n-2}-n(n+1)(1-t)^{n-1}\right ) \dt\\&>\int_{t_1}^{t_2} \ds\int_{s}^{t_2}2n\left (1-\frac{s}{t}\right )(1-t)^{n-1}\dt\\&= M_{22}(n-1, n)\\
&\ge M_{22}(i, n)
\end{align*}
In summary, setting $N_2= \lceil \frac{2}{t_{1}} \rceil - 1=6 $ and $I_2 = 3$ is sufficient to show that $M_{21}(i, n) > M_{22}(i, n)$ for all $n > N_2$ and $i \ge I_2$. Consequently, $\alpha_2(i,n)$ is decreasing as $i$ increases in this region.

We now turn to $\alpha_3$. By Lemma \ref{lm:mono}, there exists a threshold $I_3' = \left\lceil\frac{4-5t_2+\sqrt{t_2^2-8t_2+8}}{2t_2}\right\rceil=1$, such that $M_{31}(i)$ is monotonically decreasing with respect to $i$ for all $i \ge I_3'$, which means that  for all $1 \le i \le n-1$, $M_{31}(i) \ge M_{31}(n-1)$. And when \[n >\left \lceil \max_{t\in [t_2,1]}\left(\frac{2t(1-t)}{t_1t_2}\right) \right\rceil +1= 3,\] we have 
\begin{align*}
M_{31}(n-1)&=\int_{0}^{t_1}\ds \int_{t_2}^{1} \frac{t_1t_2}{t^2}\left ((n-1)n(1-t)^{n-2}-n(n+1)(1-t)^{n-1}\right )\dt \\&>\int_{0}^{t_1}\ds \int_{t_2}^{1} 2n\left (1-\frac{t_1t_2}{t^2}\right )(1-t)^{n-1}\dt\\&= M_{32}(n-1, n)\\
&\ge M_{32}(i, n).
\end{align*}
In summary, setting $N_3 =  3$ and $I_3 = 1$ is sufficient to show that $M_{31}(i) > M_{32}(i, n)$ for all $n > N_3$ and $i \ge I_3$. Consequently, $\alpha_3(i,n)$ is decreasing as $i$ increases in this region.

We now turn to $\alpha_4$. By Lemma \ref{lm:mono}, there exists a threshold $I_4' = \left\lceil\frac{4-5t_2+\sqrt{t_2^2-8t_2+8}}{2t_2}\right\rceil=1$, such that $M_{41}(i)$ is monotonically decreasing with respect to $i$ for all $i \ge I_4'$, which means that  for all $1 \le i \le n-1$, $M_{41}(i) \ge M_{41}(n-1)$. And when \[n >\left \lceil \max_{t\in [t_2,1],s\in[t_1,t_2]}\left(\frac{2t(1-t)}{st_2}\right) \right\rceil +1= 3,\] we have 
\begin{align*}
M_{41}(n-1)&=\int_{t_1}^{t_2}\ds \int_{t_2}^{1} \frac{st_2}{t^2}\left ((n-1)n(1-t)^{n-2}-n(n+1)(1-t)^{n-1}\right )\dt \\&>\int_{t_1}^{t_2}\ds \int_{t_2}^{1} 2n\left (1-\frac{st_2}{t^2}\right )(1-t)^{n-1}\dt\\&= M_{42}(n-1, n)\\
&\ge M_{42}(i, n).
\end{align*}
In summary, setting $N_4 = 3$ and $I_4 = 1$ is sufficient to show that $M_{41}(i) > M_{42}(i, n)$ for all $n > N_4$ and $i \ge I_4$. Consequently, $\alpha_4(i,n)$ is decreasing as $i$ increases in this region.

We now turn to $\alpha_5$. By Lemma \ref{lm:mono}, there exists a threshold $I_5' = \left\lceil\frac{4-5t_2+\sqrt{t_2^2-8t_2+8}}{2t_2}\right\rceil=1$, such that $M_{51}(i)$ is monotonically decreasing with respect to $i$ for all $i \ge I_5'$,
which means that  for all $1 \le i \le n-1$, $M_{51}(i) \ge M_{51}(n-1)$. 
And when \[n >\left \lceil \max_{t_2\le s\le t\le 1}\left(\frac{2t(1-t)}{s^2}\right) \right\rceil +1=3,\] we have 
\begin{align*}
M_{51}(n-1)&=\int_{t_2}^{1}\ds\int_{s}^{1} \frac{s^2}{t^2}\left ((n-1)n(1-t)^{n-2}-n(n+1)(1-t)^{n-1}\right )\dt \\&>\int_{t_2}^{1}\ds\int_{s}^{1} 2n\left (1-\frac{s^2}{t^2}\right )(1-t)^{n-1}\dt\\&= M_{52}(n-1, n)\\
&\ge M_{52}(i, n).
\end{align*}
In summary, setting $N_5 =3$ and $I_5 = 1$ is sufficient to show that $M_{51}(i) > M_{52}(i, n)$ for all $n > N_5$ and $i \ge I_5$. Consequently, $\alpha_5(i,n)$ is decreasing as $i$ increases in this region.

Our attention now shifts to $\alpha_6$. By Lemma \ref{lm:mono}, there exists a threshold $I_6' = \left\lceil\frac{6-5t_2+\sqrt{t_2^2-12t_2+12}}{2t_2}\right\rceil=3$, such that $M_{61}(i)$ is monotonically decreasing with respect to $i$ for all $i \ge I_6'$, which implies that for all $3 \le i \le n-1$, $M_{61}(i) \ge M_{61}(n-1)$. 
In order to satisfy $M_{61}(n-1) \ge M_{62}(n-1, n)$, we require the corresponding integrands to satisfy the inequality, 
i.e., \[((n+2)(n-1)t_1t_2)\cdot \frac{1}{t}+n(n+2)t\ge n(n+2)+2t_1(n-1).\]
Since $((n+2)(n-1)t_1t_2)\frac{1}{t}+n(n+2)t$ is monotonically increasing within the range of $t$ from $t_2$ to $1$, it is equivalent to \[((n+2)(n-1)t_1)+n(n+2)t_2\ge n(n+2)+2t_1(n-1),\]which yields $n\ge 7$.
In summary, setting $N_6=  6$ and $I_6 = 3$ is sufficient to show that for all $n > N_6$ and $i \ge I_6$,
\[M_{61}(i)\ge M_{61}(n-1)>M_{62}(n-1,n)\ge M_{62}(i,n).\]
Consequently, $\alpha_6(i,n)$ is decreasing as $i$ increases in this region.

We now turn to $\alpha_7$, By Lemma \ref{lm:mono}, there exists a threshold $I_7' = \left\lceil\frac{6-5t_2+\sqrt{t_2^2-12t_2+12}}{2t_2}\right\rceil=3$, such that $M_{71}(i)$ is monotonically decreasing with respect to $i$ for all $i \ge I_7'$, 
which implies that for all $3 \le i \le n-1$, $M_{71}(i) \ge M_{71}(n-1)$. 
In order to satisfy $M_{71}(n-1) \ge M_{72}(n-1, n)$, we require the corresponding integrands to satisfy the inequality,  
i.e., 
\[((n+2)(n-1) st_2)\cdot \frac{1}{t}+n(n+2)t\ge n(n+2)+2s(n-1).\]
Since $((n+2)(n-1)st_2)\frac{1}{t}+n(n+2)t-n(n+2)-2s (n-1)$ is monotonically increasing with respect to $t$ and $s$ within the range of $t$ from $t_2$ to $1$, $s$ from $t_1$ to $t_2$ when $n\ge 1$
it is equivalent to \[((n+2)(n-1)t_1)+n(n+2)t_2\ge n(n+2)+2t_1(n-1),\]which yields $n\ge 7$.
In summary, setting $N_7=  6$ and $I_7 = 3$ is sufficient to show that for all $n > N_7$ and $i \ge I_7$,
\[M_{71}(i)\ge M_{71}(n-1)>M_{72}(n-1,n)\ge M_{72}(i,n).\]
Consequently, $\alpha_7(i,n)$ is decreasing as $i$ increases in this region.

We now turn to $\alpha_8$, By Lemma \ref{lm:mono}, there exists a threshold $I_8' = \left\lceil\frac{6-5t_2+\sqrt{t_2^2-12t_2+12}}{2t_2}\right\rceil=3$, such that $M_{81}(i)$ is monotonically decreasing with respect to $i$ for all $i \ge I_8'$, 
which implies that for all $3 \le i \le n-1$, $M_{81}(i) \ge M_{81}(n-1)$. 
In order to satisfy $M_{81}(n-1) \ge M_{82}(n-1, n)$, we require the corresponding integrands to satisfy the inequality,  
i.e., 
\[((n+2)(n-1) s^2)\cdot \frac{1}{t}+n(n+2)t\ge n(n+2)+2s(n-1).\]
Since $((n+2)(n-1)s^2)\frac{1}{t}+n(n+2)t-n(n+2)-2s (n-1)$ is monotonically increasing with respect to $t$ and $s$ within the range of $t$ from $t_2$ to $1$, $s$ from $t_1$ to $t_2$ when $n\ge 1$
it is equivalent to \[((n+2)(n-1)t_2)+n(n+2)t_2\ge n(n+2)+2t_2(n-1),\]which yields $n\ge 7$.
In summary, setting $N_8=  6$ and $I_8 = 3$ is sufficient to show that for all $n > N_8$ and $i \ge I_8$,
\[M_{81}(i)\ge M_{81}(n-1)>M_{82}(n-1,n)\ge M_{82}(i,n).\]
Consequently, $\alpha_8(i,n)$ is decreasing as $i$ increases in this region.
Finally, let $N_0 = \max\{N_1, \dots, N_8\}=10$ and $I_0 = \max \{I_1, \dots, I_8\}=3$. This completes the proof of Part 1.


 \paragraph{Part 2:} Next, we discuss the relative magnitudes of $f(i, n)$ for $i \in \{1, 2, 3\}$ for all $n \in \mathbb{N}$. First, for all $n \in \mathbb{N}$, numerical computation readily shows that\[\sum_{k=1}^8 \beta_{k1}(2) - \sum_{k=1}^8 \beta_{k1}(1)\approx 0.459218>0\]
 Furthermore, it is clear that 
 \[\sum_{k=1}^8 \beta_{k2}(1, n) < \sum_{k=1}^8 \beta_{k2}(2, n).\] 
 Consequently, for all $n \in \mathbb{N}$, we have $f(1, n) < f(2, n)$.

 Similarly, numerical result indicates that 
 \[\sum_{k=1}^8 \beta_{k1}(2) > \sum_{k=1}^8 \beta_{k1}(3)\] 
 for all $n \in \mathbb{N}$. Additionally, we observe that
 \[\sum_{k=1}^8 (\beta_{k2}(3, n) - \beta_{k2}(2, n)) \le \sum_{k=1}^8 (\beta_{k2}(3, 2) - \beta_{k2}(2, 2)).\]
 Since
 \[\sum_{k=1}^8 (\beta_{k2}(3, 2) - \beta_{k2}(2, 2)) \approx 0.1049,\] 
 and 
 \[\sum_{k=1}^8 \beta_{k1}(2) - \beta_{k1}(3) \approx 0.1186,\] 
 we have
 \[\sum_{k=1}^8 \beta_{k1}(2) - \beta_{k1}(3) > \sum_{k=1}^8 (\beta_{k2}(3, 2) - \beta_{k2}(2, 2)) \ge \sum_{k=1}^8 (\beta_{k2}(3, n) - \beta_{k2}(2, n)).\]
 This confirms that for all $n \in \mathbb{N}$, we have $f(2, n) > f(3, n)$.

 \paragraph{Part 3:} Since this part only involves a finite number of values to be verified, we performed numerical computations to check these cases. We confirm that it indeed holds for all $n \le N_0$.
\end{proof}

By Equation \eqref{eq:pi1alg} and $p_{12}^{\talg}=0$, we have
\begin{equation}\label{eq:p1alg}
\begin{aligned}
p_1^{\talg}=& \int_0^{t_1} \ds \int_{t_1}^{t_2} \frac{t_1}{t}+(1-t)^{n-1}\cdot \left(1-\frac{t_1}{t}\right)  \dt
\\
&+\int_0^{t_1} \ds \int_{t_2}^{1} \frac{t_1t_2}{t^2}+(1-t)^{n-1}\cdot \left(1-\frac{t_1t_2}{t^2}\right) +(n-1)t_1 (1-t)^{n-2}\cdot \left(1- \frac{t_2}{t}\right)  \dt
\\
&+\int_{t_1}^{t_2} \ds \int_{s}^{t_2} \frac{s}{t}+(1-t)^{n-1}\cdot \left(1-\frac{s}{t}\right)    \dt
\\
&+\int_{t_1}^{t_2} \ds \int_{t_2}^{1} \frac{st_2}{t^2}+(1-t)^{n-1}\cdot \left(1-\frac{st_2}{t^2}\right) +(n-1)\cdot s \cdot (1-t)^{n-2}\cdot \left(1- \frac{t_2}{t}\right) \dt
\\
&+\int_{t_2}^{1} \ds \int_{s}^{1} \frac{s^2}{t^2}+(1-t)^{n-1}\cdot \left(1-\frac{s^2}{t^2}\right) +(n-1)\cdot s \cdot (1-t)^{n-2}\cdot \left(1- \frac{s}{t}\right) \dt.
\end{aligned}
\end{equation}

We note that $p_1^{\talg}$ depends on $n$. Let $a_n$ be the value of $p_1^{\talg}$ when there are $n$ buyers. The following lemma shows that the sequence $\{a_n\}$ is decreasing.
\begin{lemma}\label{lma:decreasen}
    For any $n\ge 1$, $a_n> a_{n+1}$.
\end{lemma}
\begin{proof}
We have
    \begin{align*}
        a_n-a_{n+1}=& \int_0^{t_1} \ds \int_{t_1}^{t_2} \Delta_1(s,t)  \dt+\int_0^{t_1} \ds \int_{t_2}^{1} \Delta_2(s,t)  \dt+\int_{t_1}^{t_2} \ds \int_{s}^{t_2} \Delta_3(s,t)   \dt\\
&+\int_{t_1}^{t_2} \ds \int_{t_2}^{1} \Delta_4(s,t) \dt+\int_{t_2}^{1} \ds \int_{s}^{1} \Delta_5(s,t) \dt,
    \end{align*}
where
\begin{align*}
    \Delta_1(s,t)&=(t-t_1)(1-t)^{n-1}> 0, \quad \forall t\in (t_1,1)\\
    \Delta_2(s,t)&=\frac{(1-t)^{n-2}}{t}\cdot \left((t^2-t_1t_2)(1-t)-t_1(t-t_2)(1-nt)\right)>0, \quad \forall t\in (t_2,1)\\
    \Delta_3(s,t)&=(t-s)(1-t)^{n-1}>0, \quad \forall t\in (s,1)\\
    \Delta_4(s,t)&=\frac{(1-t)^{n-2}}{t}\cdot \left((t^2-st_2)(1-t)-s(t-t_2)(1-nt)\right)>0, \quad \forall t\in (t_2,1)\\
    \Delta_5(s,t)&=\frac{(1-t)^{n-2}}{t}\cdot \left((t^2-s^2)(1-t)-s(t-s)(1-nt)\right)>0, \quad \forall t\in (s,1).
\end{align*}
Thus, $a_n-a_{n+1}>0$, which completes the proof. 
\end{proof}

Finally, we are now ready to analyze the competitive ratio of our algorithm.
\begin{proof}[Proof of Theorem \ref{thm:zeroupper}]
First, by Equation \eqref{eq:p1alg}, we have
\begin{align*}
\lim_{n\to \infty }p_1^{\talg} &=  \int_0^{t_1} \ds \int_{t_1}^{t_2}  \frac{t_1}{t}  \dt 
+ \int_0^{t_1} \ds \int_{t_2}^{1} \frac{t_1 t_2}{t^2}  \dt + \int_{t_1}^{t_2} \ds \int_s^{t_2} \frac{s}{t}  \dt \\
&\quad+ \int_{t_1}^{t_2} \ds \int_{t_2}^{1} \frac{st_2}{t^2}\dt + \int_{t_2}^{1} \ds \int_s^{1} \frac{s^2}{t^2}\dt\\
&=\frac{1}{12}\left (2+t_1^2(3-6t_2)+(3-2t_2)t_2^2+6t_1^2\ln\frac{t_2}{t_1}\right).
\end{align*}

Next, we observe that $\sum_{i=1}^n p_i^{\talg} = 1 - \mathbb{P}(\text{failure to sell})$. This probability of failure to sell, in turn, stems exclusively from the following three mutually exclusive cases:
\begin{enumerate}
   \item\( X_{n+1} \) arrives after both \( X_1 \) and \( X_2 \). In this scenario, any buyer who arrives after \( X_{n+1} \) is at most the third-highest bidder,  preventing the intermediary from making a sale. The probability of this event is \( \frac{1}{3} \).

    \item \( X_{n+1} \) arrives after \( X_1 \) but before \( X_2 \), and \( X_2 \) arrives before time \( t_2 \). Although \( X_2 \) is the current second-highest bidder, its arrival before \( t_2 \) still results in a failed sale. The probability of this event is \( \frac{t_2^3}{6} \).

    \item \( X_{n+1} \) arrives before \( X_1 \), with \( X_1 \) arriving before \( t_1 \) and \( X_2 \) arriving before \( t_2 \). In this situation, the intermediary still cannot  sell the asset. The probability of this event is \( \frac{t_1^2 t_2}{2} \).
\end{enumerate}
Therefore, it follows that 
\[\sum_{i=1}^n p_i^{\talg} = 1 - \left(\frac{1}{3}+\frac{t_2^3}{6}+\frac{t_1^2t_2}{2}\right).\]
Combining with Lemmas \ref{lma 4.4} and \ref{lma:decreasen}, we have 
\begin{align*}
   \max_{\bX}\frac{\opt(\bX)}{\talg(\bX)}&=\max\left\{\frac{1}{2p_1^{\talg}},\frac{1-\frac{1}{n+1}}{\sum_{i=1}^n p_i^{\talg}}\right \}\\
   &\leq \max \left\{\frac{1}{2\displaystyle\lim_{n \to \infty}{p_1^{\talg}}},\frac{1}{\sum_{i=1}^n p_i^{\talg}}\right\}\\
    &=\max\left\{\frac{6}{2+t_1^2(3-6t_2)+(3-2t_2)t_2^2+6t_1^2\ln\frac{t_2}{t_1}},\frac{1}{1 - \left (\frac{1}{3}+\frac{t_2^3}{6}+\frac{t_1^2\cdot t_2}{2}\right )}\right\}\\
    &\approx1.83683.\qedhere
\end{align*}
\end{proof}

\end{document}